\documentclass{conm-p-l}
\usepackage{amssymb}

\newtheorem{theorem}{Theorem}[section]
\newtheorem{proposition}[theorem]{Proposition}
\newtheorem{lemma}[theorem]{Lemma}
\newtheorem{corollary}[theorem]{Corollary}
\theoremstyle{definition}

\theoremstyle{remark}

\numberwithin{equation}{section}

\newcommand{\cB}{\mathcal{B}}
\newcommand{\cM}{\mathcal{M}}
\newcommand{\cO}{\mathcal{O}}
\newcommand{\cP}{\mathcal{P}}
\newcommand{\gX}{\mathfrak{X}}
\renewcommand{\gg}{\mathfrak{g}}
\newcommand{\gp}{\mathfrak{p}}
\newcommand{\gq}{\mathfrak{q}}

\makeatletter
\def\revddots{\mathinner{\mkern1mu\raise\p@
\vbox{\kern7\p@\hbox{.}}\mkern2mu
\raise4\p@\hbox{.}\mkern2mu\raise7\p@\hbox{.}\mkern1mu}}
\makeatother

\begin{document}

\title{A duality for the double fibration transform}

\author{Michael G. Eastwood}
\address{Mathematical Sciences Institute, Australian National University,
ACT 0200,\newline Australia}
\curraddr{}
\email{meastwoo@member.ams.org}
\thanks{MGE: Research supported by the Australian Research Council} 

\author{Joseph A. Wolf}
\address{Department of Mathematics, University of California,
Berkeley, CA 94720--3840, USA}
\curraddr{}
\email{jawolf@math.berkeley.edu}
\thanks{JAW: Research partially supported by NSF Grant DMS 99-88643}

\subjclass[2010]{Primary 32L25; Secondary 22E46, 32L10}
\date{}

\begin{abstract}
We establish a duality within the spectral sequence that governs the
holomorphic double fibration transform. It has immediate application to the
questions of injectivity and range characterization for this transform. We
discuss some key examples and an improved duality that holds in the Hermitian
holomorphic case.
\end{abstract}

\maketitle

\section{Double fibrations}\label{double_fibrations}
In this article we shall always work in the holomorphic category. By a 
{\em double fibration\/} we shall mean a diagram of the form
\begin{equation}\label{gen_doublefibration}
\raisebox{-20pt}{\begin{picture}(100,40)
\put(20,5){\makebox(0,0){$D$}}
\put(50,35){\makebox(0,0){$\gX_D$}}
\put(80,5){\makebox(0,0){$\cM_D$}}
\put(43,28){\vector(-1,-1){16}}
\put(57,28){\vector(1,-1){16}}
\put(33,23){\makebox(0,0){$\mu$}}
\put(68,23){\makebox(0,0){$\nu$}}
\end{picture}}
\end{equation}
where
\begin{equation}\label{conditions}
\begin{minipage}{320pt}\mbox{}\vspace{-10pt}
\begin{itemize}
\item $D$, $\gX_D$, and $\cM_D$ are complex manifolds;
\item $\mu$ is a holomorphic submersion with contractible fibres;
\item $\nu$ is a holomorphic submersion with compact fibres;
\item $\gX_D\xrightarrow{\,(\mu,\nu)\,}D\times\cM_D$ is a holomorphic 
      embedding;
\item $\cM$ is a contractible Stein manifold.
\end{itemize}
\end{minipage}
\end{equation}
Examples of double fibrations arise naturally as follows. Let $G$ be a complex
semisimple (or even reductive) Lie group. There is a beautiful class of complex
homogeneous spaces $Z = G/Q$ that can be characterized by any of the following
equivalent conditions (see e.g.~\cite{FHW} for details).
\begin{itemize}
\item $Z$ is a compact complex manifold; 
\item $Z$ is a compact K\"ahler manifold;
\item $Z$ is a complex projective variety;
\item $Q$ is a parabolic subgroup of $G$. 
\end{itemize}
We shall refer to such compact complex homogeneous spaces $Z$ as {\em complex
flag manifolds}. Now fix a complex flag manifold $Z=G/Q$ and consider a real
form $G_0$ of~$G$. Then it is known~\cite{W} that the natural action of $G_0$
on $Z$ has only finitely many orbits and so there is at least one open orbit.
If $G_0$ is compact, then it acts transitively on $Z$ and there are few other 
exceptional cases when this happens. Otherwise, an open $G_0$-orbit 
$D\subsetneq Z$ is known as a {\em flag domain\/}. As a simple example, let 
us take $G={\mathrm{SL}}(4,{\mathbb{C}})$ acting on $Z={\mathbb{CP}}_3$ in the 
usual fashion, namely 
$${\mathrm{SL}}(4,{\mathbb{C}})\times{\mathbb{CP}}_3\ni(A,[z])\mapsto[Az]\in
{\mathbb{CP}}_3,$$
where $z\in{\mathbb{C}}^4$ is regarded as a column vector. If we take
$$\left[\begin{matrix}\ast\\ 0\\0 \\0\end{matrix}\right]\in{\mathbb{CP}}_3
\enskip\mbox{as basepoint, then}\enskip
Q=\left\{\left[\begin{matrix}\ast&\ast&\ast&\ast\\
0&\ast&\ast&\ast\\
0&\ast&\ast&\ast\\
0&\ast&\ast&\ast\end{matrix}\right]\in{\mathrm{SL}}(4,{\mathbb{C}})\right\}.$$
If we take $G_0={\mathrm{SU}}(2,2)$, defined as preserving the Hermitian form
\begin{equation}\label{hermitian}
\langle w,z\rangle\equiv w_1\overline{z_1}+w_2\overline{z_2}
-w_3\overline{z_3}-w_4\overline{z_4}\end{equation}
on~${\mathbb{C}}^4$, then
$$D={\mathbb{CP}}_3^+
\equiv
\{[z]\in{\mathbb{CP}}_3\mid|z_1|^2+|z_2|^2-|z_3|^2-|z_4|^2>0\}$$
is a flag domain for the action of $G_0$ on~$Z$. 

In general, fixing $K_0\subset G_0$ a maximal compact subgroup, it is known
\cite{W} that there is just one $K_0$-orbit $C_0$ in~$D$ that is a complex
submanifold of~$Z$. We regard $C_0$ as the basepoint of the {\em cycle space\/}
$$\cM_D\equiv\mbox{connected component of $C_0$ in }
\{gC_0\mid g\in G\mbox{ and }gC_0\subset D\}$$
of~$D$. Evidently, $\cM_D$ is an open subset of 
$$\cM_Z\equiv\{gC_0\mid g\in G\}=G/J,\enskip\mbox{ where }
J\equiv\{g\in G\mid gC_0=C_0\}$$
and hence is a complex manifold. Let us set
$$\gX_Z\equiv G/(Q\cap J)\quad\mbox{and}\quad 
\gX_D\equiv\{(z,C)\in D\times\cM_D\mid z\in C\}.$$
Then 
\begin{equation}\label{holistic}\raisebox{-20pt}{\begin{picture}(100,40)
\put(20,5){\makebox(0,0){$D$}}
\put(50,35){\makebox(0,0){$\gX_D$}}
\put(80,5){\makebox(0,0){$\cM_D$}}
\put(43,28){\vector(-1,-1){16}}
\put(57,28){\vector(1,-1){16}}
\put(33,23){\makebox(0,0){$\mu$}}
\put(68,23){\makebox(0,0){$\nu$}}
\end{picture}}{}^{\mathrm{open}}\subset
\raisebox{-20pt}{\begin{picture}(100,40)
\put(20,5){\makebox(0,0){$Z$}}
\put(50,35){\makebox(0,0){$\gX_Z$}}
\put(80,5){\makebox(0,0){$\cM_Z$}}
\put(43,28){\vector(-1,-1){16}}
\put(57,28){\vector(1,-1){16}}
\end{picture}}\end{equation}
and it is known for any flag domain (see e.g.~\cite{FHW} for details) that all
the conditions (\ref{conditions}) of a double fibration are satisfied.

In our example, we may take 
\begin{equation}\label{maxcompact}
K_0={\mathrm{S}}({\mathrm{U}}(2)\times{\mathrm{U}}(2))=
\left\{\left[\begin{matrix}\ast&\ast&0&0\\
\ast&\ast&0&0\\
0&0&\ast&\ast\\
0&0&\ast&\ast\end{matrix}\right]\in{\mathrm{SU}}(2,2)\right\}\end{equation}
whence 
$$C_0=\left\{\left[\begin{matrix}\ast\\ \ast\\ 0\\ 0
\end{matrix}\right]\in{\mathbb{CP}}_3\right\},\quad
J=\left\{\left[\begin{matrix}\ast&\ast&\ast&\ast\\
\ast&\ast&\ast&\ast\\
0&0&\ast&\ast\\
0&0&\ast&\ast\end{matrix}\right]\in{\mathrm{SL}}(4,{\mathbb{C}})\right\},$$
and $\cM_Z={\mathrm{Gr}}_2({\mathbb{C}}^4)$, the Grassmannian of $2$-planes 
in~${\mathbb{C}}^4$. The base cycle $C_0$ and, therefore, every other cycle is 
intrinsically a Riemann sphere~${\mathbb{CP}}_1$. Geometrically, 
$$\cM_D=\{\Pi\in{\mathrm{Gr}}_2({\mathbb{C}}^4)\mid
\langle\underbar{\enskip},\underbar{\enskip}\rangle|_\Pi\mbox{ is
positive definite}\}\equiv{\mathbb{M}}^{++}$$
and analytically we may realize $\cM_D$ as a convex tube domain 
in~${\mathbb{C}}^4$
$$\cM_D\cong\{\zeta=x+iy\in{\mathbb{C}}^4\mid
x_1{}^2>x_2{}^2+x_3{}^2+x_4{}^2\mbox{ and }x_1>0\}$$
by means of
$${\mathbb{C}}^4\ni(\zeta_1,\zeta_2,\zeta_3,\zeta_4)\longmapsto\Pi\equiv
{\mathrm{span}}\left\{\left[\begin{matrix}
1+\zeta_1+\zeta_2\\ \zeta_3+i\zeta_4\\ 1-\zeta_1-\zeta_2\\ -\zeta_3-i\zeta_4
\end{matrix}\right],\left[\begin{matrix}
\zeta_3-i\zeta_4\\ 1+\zeta_1-\zeta_2\\ -\zeta_3+i\zeta_4\\ 1-\zeta_1+\zeta_2
\end{matrix}\right]
\right\}.$$
Notice that, in this particular case, the cycle space $\cM_D$ is itself a flag 
domain (for the action of ${\mathrm{SU}}(2,2)$ on 
${\mathrm{Gr}}_2({\mathbb{C}}^4)$). This is unusual.

For our second example, let us start with another of the open orbits of 
${\mathrm{SU}}(2,2)$ on ${\mathrm{Gr}}_2({\mathbb{C}}^4)$, namely 
$$D=\{\Pi\in{\mathrm{Gr}}_2({\mathbb{C}}^4)\mid
\langle\underbar{\enskip},\underbar{\enskip}\rangle|_\Pi\mbox{ is
strictly indefinite}\}\equiv{\mathbb{M}}^{+-}.$$
With the same choice (\ref{maxcompact}) of maximal compact subgroup~$K_0$, the 
base cycle $C_0$ is
$$\left\{\Pi\in{\mathrm{Gr}}_2({\mathbb{C}}^4)\mid
\Pi={\mathrm{span}}\{\alpha,\beta\}\mbox{ for some }\Big\{\!\begin{array}{l}
\alpha\mbox{ of the form }[\ast,\ast,0,0]^t\\
\beta\mbox{ of the form }[0,0,\ast,\ast]^t
\end{array}\right\}.$$
Hence the base cycle and, therefore, every other cycle is intrinsically
${\mathbb{CP}}_1\times{\mathbb{CP}}_1$. By definition, we always have
$J\supseteq K$, the complexification of~$K_0$, but often they are equal this is
the case here. The cycle space $\cM_D$ is
${\mathbb{M}}^{++}\times{\mathbb{M}}^{--}$, where ${\mathbb{M}}^{--}$ denotes
the set of planes in ${\mathbb{C}}^4$ on which
$\langle\underbar{\enskip},\underbar{\enskip}\rangle$ is negative definite. As
a product of two Stein manifolds it is Stein. For $(\Pi_1,\Pi_2)\in\cM_D$, the
corresponding cycle is
$$\big\{\Pi\in{\mathrm{Gr}}_2({\mathbb{C}}^4)\mid
\Pi={\mathrm{span}}\{\alpha,\beta\}\mbox{ for some }\alpha\in\Pi_1\mbox{ and }
\beta\in\Pi_2\big\}.$$
\section{The transform}
Consider a general double fibration (\ref{gen_doublefibration}), satisfying the
conditions~(\ref{conditions}), and suppose we are given a holomorphic vector
bundle $E$ on~$D$ and a cohomology class $\omega\in H^r(D;\cO(E))$. We shall
continue to refer to the compact complex submanifolds $\mu(\nu^{-1}(x))$ for
$x\in\cM_D$ as {\em cycles\/} in~$D$ and now consider the restriction of 
$\omega$ to these cycles: 
$$\omega|_{\mu(\nu^{-1}(x))}\in 
H^r\big(\mu(\nu^{-1}(x));\cO(E|_{\mu(\nu^{-1}(x))})\big),\enskip
\mbox{as $x\in\cM_D$ varies}.$$ 
As $\nu$ has compact fibers, these cohomology spaces are finite-dimensional and
we shall suppose that their dimension is constant as $x\in\cM_D$ varies
(generically this is the case and in the homogeneous setting, as discussed
above, this manifest if one starts with $E$ a $G$-homogeneous vector bundle).
Then, as $x\in\cM_D$ varies we obtain a vector bundle $E^\prime$ on $\cM_D$ and
a holomorphic section $\cP\omega\in\Gamma(\cM_D,\cO(E^\prime))$ thereof. This
is the {\em double fibration transform\/} of~$\omega$. It is often most
interesting starting with cohomology in the same degree as the dimension of the
fibers of~$\nu$. Two natural questions associated with this transform are
\begin{itemize}
\item is it injective?
\item what is its range?
\end{itemize}
There are clear parallels with the Radon transform and other transforms from
real integral geometry, especially when integrating over compact cycles.

The complex version, however, benefits from the following general result.
\begin{theorem}\label{from_BE}
For any double fibration\/ {\rm(\ref{gen_doublefibration})}, 
and holomorphic vector bundle $E$ on~$D$, there is a spectral sequence
\begin{equation}\label{ss}
E_1^{p,q}=\Gamma(\cM_D;\nu_*^q\Omega_\mu^p(E))\Longrightarrow 
H^{p+q}(D;\cO(E)),\end{equation}
where 
\begin{itemize}
\item $\Omega_\mu^1\equiv\Omega_{\gX_D}^1/\mu^*\Omega_D^1$, 
the holomorphic $1$-forms along the fibers of $\mu$; 
\item $\Omega_\mu^p\equiv\Lambda^p\Omega_\mu^1$,
the holomorphic $p$-forms along the fibers of $\mu$;
\item $\Omega_\mu^p(E)\equiv\Omega_\mu^p\otimes\mu^*E$.
\end{itemize}
\end{theorem}
\begin{proof} There are two stages to the proof, the details of which may be
found in~\cite{BE}. The first uses that the fibers of $\mu$ are contractible
to conclude that
$$H^r(D;\cO(E))\cong H^r(\gX_D;\mu^{-1}\cO(E))$$
where $\mu^{-1}\cO(E)$ denotes the sheaf of germs of holomorphic sections of 
$\mu^*E$ on $\gX_D$ that are locally constant along the fibers of~$\mu$. The 
second stage uses the resolution $0\to\mu^{-1}\cO(E)\to\Omega_\mu^\bullet(E)$ 
to construct a spectral sequence
$$E_1^{p,q}=H^q(\gX_D;\Omega_\mu^p(E))\Longrightarrow 
H^{p+q}(\gX_D;\mu^{-1}\cO(E)),$$
which combines with the natural isomorphisms
$H^q(\gX_D,\cO(F))\cong\Gamma(\cM_D,\nu_*^q\cO(F))$, valid for any holomorphic
vector bundle $F$ on $\gX_D$ because $\cM_D$ is Stein. \end{proof} 
For the rest of this article we shall suppose that the direct images
$\nu_*^q\Omega_\mu^p(E)$ are locally free and therefore may be regarded as
holomorphic vector bundles on~$\cM_D$. {From} this viewpoint, the
$E_1$-differentials become first order differential operators on $\cM_D$ and,
more generally, the spectral sequence ideally interprets the cohomology
$H^r(D;\cO(E))$ in terms of systems of holomorphic differential equations
on~$\cM_D$. This is especially interesting when $D$ is a flag domain, $\cM_D$ 
is its cycle space, and $E$ is $G$-homogeneous because then this double 
fibration transform can provide useful alternative realizations of the 
$G_0$-representations afforded by $H^r(D;\cO(E))$. 

\section{Examples}
Let us now return to the flag domains introduced in \S\ref{double_fibrations}
and see how the spectral sequence (\ref{ss}) works out for the first of these
domains, namely $D={\mathbb{CP}}_3^+$. The main issue in executing (\ref{ss})
is in computing the direct images $\nu_*^q\Omega_\mu^p(E)$. We need a notation
for the irreducible homogeneous vector bundles on the flag manifold~$Z$. For
this we shall follow~\cite{BE}, recording both the parabolic subgroup $Q$ and
the representation of $Q$ by annotating the appropriate Dynkin diagram (it
turns out to be most convenient to record the lowest weight of the
representation). For our first domain, in which $Z={\mathbb{CP}}_3$, the
irreducible homogeneous vector bundles are
$$\begin{picture}(42,5)
\put(20,1.5){\makebox(0,0){$\bullet$}}
\put(36,1.5){\makebox(0,0){$\bullet$}}
\put(4,1.5){\makebox(0,0){$\times$}}
\put(4,1.5){\line(1,0){32}}
\put(4,6){\makebox(0,0)[b]{$\scriptstyle a$}}
\put(20,6){\makebox(0,0)[b]{$\scriptstyle b$}}
\put(36,6){\makebox(0,0)[b]{$\scriptstyle c$}}
\end{picture}\quad\mbox{for integers }a,b,c\mbox{ with }b,c\geq 0.$$
The details are in \cite{BE} but some particular cases are
$$\begin{array}{l}
\begin{picture}(42,5)
\put(20,1.5){\makebox(0,0){$\bullet$}}
\put(36,1.5){\makebox(0,0){$\bullet$}}
\put(4,1.5){\makebox(0,0){$\times$}}
\put(4,1.5){\line(1,0){32}}
\put(4,6){\makebox(0,0)[b]{$\scriptstyle 0$}}
\put(20,6){\makebox(0,0)[b]{$\scriptstyle 0$}}
\put(36,6){\makebox(0,0)[b]{$\scriptstyle 0$}}
\end{picture}=\mbox{the trivial bundle}\equiv\cO\\
\begin{picture}(42,12)
\put(20,1.5){\makebox(0,0){$\bullet$}}
\put(36,1.5){\makebox(0,0){$\bullet$}}
\put(4,1.5){\makebox(0,0){$\times$}}
\put(4,1.5){\line(1,0){32}}
\put(4,6){\makebox(0,0)[b]{$\scriptstyle 1$}}
\put(20,6){\makebox(0,0)[b]{$\scriptstyle 0$}}
\put(36,6){\makebox(0,0)[b]{$\scriptstyle 1$}}
\end{picture}=\mbox{the holomorphic tangent bundle}\equiv\Theta\\
\begin{picture}(42,12)
\put(20,1.5){\makebox(0,0){$\bullet$}}
\put(36,1.5){\makebox(0,0){$\bullet$}}
\put(4,1.5){\makebox(0,0){$\times$}}
\put(4,1.5){\line(1,0){32}}
\put(4,5){\makebox(0,0)[b]{$\scriptstyle -2$}}
\put(20,6){\makebox(0,0)[b]{$\scriptstyle 1$}}
\put(36,6){\makebox(0,0)[b]{$\scriptstyle 0$}}
\end{picture}=\mbox{the holomorphic cotangent bundle}\equiv\Omega^1\\
\begin{picture}(42,12)
\put(20,1.5){\makebox(0,0){$\bullet$}}
\put(36,1.5){\makebox(0,0){$\bullet$}}
\put(4,1.5){\makebox(0,0){$\times$}}
\put(4,1.5){\line(1,0){32}}
\put(4,5){\makebox(0,0)[b]{$\scriptstyle -3$}}
\put(20,6){\makebox(0,0)[b]{$\scriptstyle 0$}}
\put(36,6){\makebox(0,0)[b]{$\scriptstyle 1$}}
\end{picture}=\mbox{the bundle of holomorphic $2$-forms}\equiv\Omega^2\\
\begin{picture}(42,12)
\put(20,1.5){\makebox(0,0){$\bullet$}}
\put(36,1.5){\makebox(0,0){$\bullet$}}
\put(4,1.5){\makebox(0,0){$\times$}}
\put(4,1.5){\line(1,0){32}}
\put(4,6){\makebox(0,0)[b]{$\scriptstyle 1$}}
\put(20,6){\makebox(0,0)[b]{$\scriptstyle 0$}}
\put(36,6){\makebox(0,0)[b]{$\scriptstyle 0$}}
\end{picture}=\mbox{the tautological bundle}\equiv\cO(1)\\
\begin{picture}(42,12)
\put(20,1.5){\makebox(0,0){$\bullet$}}
\put(36,1.5){\makebox(0,0){$\bullet$}}
\put(4,1.5){\makebox(0,0){$\times$}}
\put(4,1.5){\line(1,0){32}}
\put(4,6){\makebox(0,0)[b]{$\scriptstyle k$}}
\put(20,6){\makebox(0,0)[b]{$\scriptstyle 0$}}
\put(36,6){\makebox(0,0)[b]{$\scriptstyle 0$}}
\end{picture}=
\mbox{the $k^{\mathrm{th}}$ power of the tautological bundle}\equiv\cO(k).
\end{array}$$
Similarly, the irreducible homogeneous vector bundles on $\gX_Z$, the flag 
manifold
$${\mathrm{F}}_{1,2}({\mathbb{C}}^4)=
\{(L,\Pi)\in{\mathbb{CP}}_3\times{\mathrm{Gr}}_2({\mathbb{C}}^4)\mid 
L\subset\Pi\}$$
are given by
$$\begin{picture}(42,5)
\put(20,1.5){\makebox(0,0){$\times$}}
\put(36,1.5){\makebox(0,0){$\bullet$}}
\put(4,1.5){\makebox(0,0){$\times$}}
\put(4,1.5){\line(1,0){32}}
\put(4,6){\makebox(0,0)[b]{$\scriptstyle a$}}
\put(20,6){\makebox(0,0)[b]{$\scriptstyle b$}}
\put(36,6){\makebox(0,0)[b]{$\scriptstyle c$}}
\end{picture}\quad\mbox{for integers }a,b,c\mbox{ with }c\geq 0.$$

For computational purposes, it is always better to consider the diagram
\begin{equation}\label{bigger}
\raisebox{-20pt}{\begin{picture}(100,40)
\put(5,5){\makebox(0,0){$G/Q=Z$}}
\put(82,35){\makebox(0,0){$\gX_Z=G/(Q\cap J)$}}
\put(95,5){\makebox(0,0){$\cM_Z=G/J$}}
\put(43,28){\vector(-1,-1){16}}
\put(57,28){\vector(1,-1){16}}
\put(33,23){\makebox(0,0){$\mu$}}
\put(68,23){\makebox(0,0){$\nu$}}
\end{picture}}\end{equation}
from (\ref{holistic}), where we have extended the definition of $\mu$ and $\nu$
to $\gX_Z$ as shown. The point is that, with this enhanced definition of $\nu$,
we have $\gX_D=\nu^{-1}(\cM_D)$ and so the fibers over $\cM_D$ are unchanged.
In particular, the direct images~$\nu_*^q\Omega_\mu^p(E)$, as required in the
spectral sequence~(\ref{ss}), can be computed from (\ref{bigger}) and then
simply restricted to the open Stein subset $\cM_D\subset\cM_Z$. The advantage
of (\ref{bigger}) is that all three spaces are $G$-homogeneous and the two
mappings are $G$-equivariant. Hence, we may use representation theory to
compute $\nu_*^q\Omega_\mu^p(E)$ et cetera.

With this enhanced viewpoint in place, the bundles of holomorphic forms along 
the fibers of $\mu$ are
\begin{equation}\label{relative_forms}\Omega_\mu^0=\begin{picture}(42,5)
\put(20,1.5){\makebox(0,0){$\times$}}
\put(36,1.5){\makebox(0,0){$\bullet$}}
\put(4,1.5){\makebox(0,0){$\times$}}
\put(4,1.5){\line(1,0){32}}
\put(4,6){\makebox(0,0)[b]{$\scriptstyle 0$}}
\put(20,6){\makebox(0,0)[b]{$\scriptstyle 0$}}
\put(36,6){\makebox(0,0)[b]{$\scriptstyle 0$}}
\end{picture},\qquad
\Omega_\mu^1=\begin{picture}(42,5)
\put(20,1.5){\makebox(0,0){$\times$}}
\put(36,1.5){\makebox(0,0){$\bullet$}}
\put(4,1.5){\makebox(0,0){$\times$}}
\put(4,1.5){\line(1,0){32}}
\put(4,6){\makebox(0,0)[b]{$\scriptstyle 1$}}
\put(20,5){\makebox(0,0)[b]{$\scriptstyle -2$}}
\put(36,6){\makebox(0,0)[b]{$\scriptstyle 1$}}
\end{picture},\qquad
\Omega_\mu^2=\begin{picture}(42,5)
\put(20,1.5){\makebox(0,0){$\times$}}
\put(36,1.5){\makebox(0,0){$\bullet$}}
\put(4,1.5){\makebox(0,0){$\times$}}
\put(4,1.5){\line(1,0){32}}
\put(4,6){\makebox(0,0)[b]{$\scriptstyle 2$}}
\put(20,5){\makebox(0,0)[b]{$\scriptstyle -3$}}
\put(36,6){\makebox(0,0)[b]{$\scriptstyle 0$}}
\end{picture}.\end{equation}
Now let us consider the double fibration transform for $H^r(D;\cO(k))$. Line 
bundles are straightforward because
$$\mu^*\cO(k)=\mu^*\begin{picture}(42,12)
\put(20,1.5){\makebox(0,0){$\bullet$}}
\put(36,1.5){\makebox(0,0){$\bullet$}}
\put(4,1.5){\makebox(0,0){$\times$}}
\put(4,1.5){\line(1,0){32}}
\put(4,6){\makebox(0,0)[b]{$\scriptstyle k$}}
\put(20,6){\makebox(0,0)[b]{$\scriptstyle 0$}}
\put(36,6){\makebox(0,0)[b]{$\scriptstyle 0$}}
\end{picture}=
\begin{picture}(42,12)
\put(20,1.5){\makebox(0,0){$\times$}}
\put(36,1.5){\makebox(0,0){$\bullet$}}
\put(4,1.5){\makebox(0,0){$\times$}}
\put(4,1.5){\line(1,0){32}}
\put(4,6){\makebox(0,0)[b]{$\scriptstyle k$}}
\put(20,6){\makebox(0,0)[b]{$\scriptstyle 0$}}
\put(36,6){\makebox(0,0)[b]{$\scriptstyle 0$}}
\end{picture},$$
which is irreducible.
Writing $\Omega_\mu^p(k)$ for $\Omega_\mu^p\otimes\mu^*\cO(k)$, we have
\begin{equation}\label{relativedeRham}\Omega_\mu^0(k)=\begin{picture}(42,12)
\put(20,1.5){\makebox(0,0){$\times$}}
\put(36,1.5){\makebox(0,0){$\bullet$}}
\put(4,1.5){\makebox(0,0){$\times$}}
\put(4,1.5){\line(1,0){32}}
\put(4,6){\makebox(0,0)[b]{$\scriptstyle k$}}
\put(20,6){\makebox(0,0)[b]{$\scriptstyle 0$}}
\put(36,6){\makebox(0,0)[b]{$\scriptstyle 0$}}
\end{picture}\qquad
\Omega_\mu^1(k)=\begin{picture}(42,12)
\put(20,1.5){\makebox(0,0){$\times$}}
\put(36,1.5){\makebox(0,0){$\bullet$}}
\put(4,1.5){\makebox(0,0){$\times$}}
\put(4,1.5){\line(1,0){32}}
\put(4,5){\makebox(0,0)[b]{$\scriptstyle k+1$}}
\put(20,5){\makebox(0,0)[b]{$\scriptstyle -2$}}
\put(36,6){\makebox(0,0)[b]{$\scriptstyle 1$}}
\end{picture}\qquad
\Omega_\mu^2(k)=\begin{picture}(42,12)
\put(20,1.5){\makebox(0,0){$\times$}}
\put(36,1.5){\makebox(0,0){$\bullet$}}
\put(4,1.5){\makebox(0,0){$\times$}}
\put(4,1.5){\line(1,0){32}}
\put(4,5){\makebox(0,0)[b]{$\scriptstyle k+2$}}
\put(20,5){\makebox(0,0)[b]{$\scriptstyle -3$}}
\put(36,6){\makebox(0,0)[b]{$\scriptstyle 0$}}
\end{picture}.\end{equation}
The direct images are computed in accordance with the Bott-Borel-Weil Theorem 
along the fibers of~$\nu$, which reads
\begin{equation}\label{simple_BBW}\begin{array}{l}
\nu_*\begin{picture}(42,12)
\put(20,1.5){\makebox(0,0){$\times$}}
\put(36,1.5){\makebox(0,0){$\bullet$}}
\put(4,1.5){\makebox(0,0){$\times$}}
\put(4,1.5){\line(1,0){32}}
\put(4,6){\makebox(0,0)[b]{$\scriptstyle a$}}
\put(20,6){\makebox(0,0)[b]{$\scriptstyle b$}}
\put(36,6){\makebox(0,0)[b]{$\scriptstyle c$}}
\end{picture}=
\begin{picture}(42,12)
\put(20,1.5){\makebox(0,0){$\times$}}
\put(36,1.5){\makebox(0,0){$\bullet$}}
\put(4,1.5){\makebox(0,0){$\bullet$}}
\put(4,1.5){\line(1,0){32}}
\put(4,6){\makebox(0,0)[b]{$\scriptstyle a$}}
\put(20,6){\makebox(0,0)[b]{$\scriptstyle b$}}
\put(36,6){\makebox(0,0)[b]{$\scriptstyle c$}}
\end{picture}\quad\mbox{for }a\geq 0\\
\nu_*^1\begin{picture}(42,12)
\put(20,1.5){\makebox(0,0){$\times$}}
\put(36,1.5){\makebox(0,0){$\bullet$}}
\put(4,1.5){\makebox(0,0){$\times$}}
\put(4,1.5){\line(1,0){32}}
\put(4,6){\makebox(0,0)[b]{$\scriptstyle a$}}
\put(20,6){\makebox(0,0)[b]{$\scriptstyle b$}}
\put(36,6){\makebox(0,0)[b]{$\scriptstyle c$}}
\end{picture}=
\begin{picture}(60,12)
\put(34,1.5){\makebox(0,0){$\times$}}
\put(54,1.5){\makebox(0,0){$\bullet$}}
\put(4,1.5){\makebox(0,0){$\bullet$}}
\put(4,1.5){\line(1,0){50}}
\put(4,6){\makebox(0,0)[b]{$\scriptstyle -a-2$}}
\put(34,6){\makebox(0,0)[b]{$\scriptstyle a+b+1$}}
\put(54,6){\makebox(0,0)[b]{$\scriptstyle c$}}
\end{picture}\quad\mbox{for }a\leq -2,
\end{array}\end{equation}
with all other direct images vanishing (see e.g.~\cite{BE} for details). In 
particular, there are the spectral sequences (\ref{ss}) of the following form.
$$\begin{picture}(300,60)(7,-3)
\put(60,45){\framebox{if $k\geq 0$}}
\put(250,45){\framebox{if $k\leq-4$}}
\put(22,0){\line(1,0){12}}
\put(77,0){\line(1,0){12}}
\put(132,0){\vector(1,0){15}}
\put(147,-5){$\scriptstyle p$}
\put(0,11){\line(0,1){8}}
\put(0,31){\vector(0,1){12}}
\put(-7,42){$\scriptstyle q$}
\put(-2,22){$0$}
\put(53,22){$0$}
\put(108,22){$0$}
\put(-19,-2){\begin{picture}(42,12)
\put(20,1.5){\makebox(0,0){$\times$}}
\put(36,1.5){\makebox(0,0){$\bullet$}}
\put(4,1.5){\makebox(0,0){$\bullet$}}
\put(4,1.5){\line(1,0){32}}
\put(4,6){\makebox(0,0)[b]{$\scriptstyle k$}}
\put(20,6){\makebox(0,0)[b]{$\scriptstyle 0$}}
\put(36,6){\makebox(0,0)[b]{$\scriptstyle 0$}}
\end{picture}}
\put(36,-2){\begin{picture}(42,12)
\put(20,1.5){\makebox(0,0){$\times$}}
\put(36,1.5){\makebox(0,0){$\bullet$}}
\put(4,1.5){\makebox(0,0){$\bullet$}}
\put(4,1.5){\line(1,0){32}}
\put(4,5){\makebox(0,0)[b]{$\scriptstyle k+1$}}
\put(20,5){\makebox(0,0)[b]{$\scriptstyle -2$}}
\put(36,6){\makebox(0,0)[b]{$\scriptstyle 1$}}
\end{picture}}
\put(91,-2){\begin{picture}(42,12)
\put(20,1.5){\makebox(0,0){$\times$}}
\put(36,1.5){\makebox(0,0){$\bullet$}}
\put(4,1.5){\makebox(0,0){$\bullet$}}
\put(4,1.5){\line(1,0){32}}
\put(4,5){\makebox(0,0)[b]{$\scriptstyle k+2$}}
\put(20,5){\makebox(0,0)[b]{$\scriptstyle -3$}}
\put(36,6){\makebox(0,0)[b]{$\scriptstyle 0$}}
\end{picture}}
\put(185,0){\line(1,0){57}}
\put(250,0){\line(1,0){57}}
\put(315,0){\vector(1,0){15}}
\put(330,-5){$\scriptstyle p$}
\put(180,7){\line(0,1){10}}
\put(180,33){\vector(0,1){10}}
\put(173,42){$\scriptstyle q$}
\put(178,-2){$0$}
\put(243,-2){$0$}
\put(308,-2){$0$}
\put(161,20){\begin{picture}(42,12)
\put(20,1.5){\makebox(0,0){$\times$}}
\put(36,1.5){\makebox(0,0){$\bullet$}}
\put(-2,1.5){\makebox(0,0){$\bullet$}}
\put(-2,1.5){\line(1,0){38}}
\put(-6,5){\makebox(0,0)[b]{$\scriptstyle -k-2$}}
\put(20,5){\makebox(0,0)[b]{$\scriptstyle k+1$}}
\put(36,6){\makebox(0,0)[b]{$\scriptstyle 0$}}
\end{picture}}
\put(226,20){\begin{picture}(42,12)
\put(20,1.5){\makebox(0,0){$\times$}}
\put(36,1.5){\makebox(0,0){$\bullet$}}
\put(1,1.5){\makebox(0,0){$\bullet$}}
\put(1,1.5){\line(1,0){35}}
\put(-3,5){\makebox(0,0)[b]{$\scriptstyle -k-3$}}
\put(20,5){\makebox(0,0)[b]{$\scriptstyle k$}}
\put(36,6){\makebox(0,0)[b]{$\scriptstyle 1$}}
\end{picture}}
\put(291,20){\begin{picture}(42,12)
\put(20,1.5){\makebox(0,0){$\times$}}
\put(36,1.5){\makebox(0,0){$\bullet$}}
\put(1,1.5){\makebox(0,0){$\bullet$}}
\put(1,1.5){\line(1,0){35}}
\put(-3,5){\makebox(0,0)[b]{$\scriptstyle -k-4$}}
\put(20,5){\makebox(0,0)[b]{$\scriptstyle k$}}
\put(36,6){\makebox(0,0)[b]{$\scriptstyle 0$}}
\end{picture}}
\end{picture}$$
Let us say that the spectral sequence $E_1^{p,q}$ is 
{\em concentrated in degree zero\/} if and only if $E_1^{p,q}=0$ for $q>0$ 
and {\em strictly concentrated in degree zero\/} if and only if, in addition, 
$E_1^{p,0}\not=0,\;\forall p$. Similarly, let us say we have 
{\em concentration in top degree\/} if and only if $E_1^{p,q}=0$ for $q<s$, 
where $s=\dim_{\mathbb{C}}(\mbox{fibres of }\nu)$ and {\em strictly\/} so if 
and only if $E_1^{p,s}\not=0,\;\forall p$. Thus, strict concentration occurs 
in this example for $k\geq 0$ or $k\leq -4$. In fact, it is easily verified 
that
\begin{equation}\label{sampleconcentration}\begin{array}{rcl}
k\geq 0&\implies&\mbox{strict concentration in degree zero},\\
k=-1&\implies&\mbox{concentration in degree zero},\\
k=-2&\implies&\mbox{no concentration},\\
k=-3&\implies&\mbox{concentration in top degree }(s=1),\\
k\leq -4&\implies&\mbox{strict concentration in top degree}.
\end{array}\end{equation}
The double fibration transform in this case is known as the {\em Penrose 
transform\/}~\cite{EPW}. Always, the spectral sequence is most easily 
interpreted when it concentrates in top degree for then it collapses to yield, 
in particular, an isomorphism
$${\mathcal{P}}:H^s(D;\cO(E))\xrightarrow{\!\simeq\enskip}\ker:
\Gamma(\cM_D;\nu_*^s\Omega_\mu^0(E))\to\Gamma(\cM_D;\nu_*^s\Omega_\mu^1(E)).$$
In our example
$${\mathcal{P}}:H^s(D;\cO(k))\xrightarrow{\!\simeq\enskip}\ker:
\Gamma(\cM_D;\quad\enskip\begin{picture}(42,12)
\put(20,1.5){\makebox(0,0){$\times$}}
\put(36,1.5){\makebox(0,0){$\bullet$}}
\put(-2,1.5){\makebox(0,0){$\bullet$}}
\put(-2,1.5){\line(1,0){38}}
\put(-6,5){\makebox(0,0)[b]{$\scriptstyle -k-2$}}
\put(20,5){\makebox(0,0)[b]{$\scriptstyle k+1$}}
\put(36,6){\makebox(0,0)[b]{$\scriptstyle 0$}}
\end{picture})\to\Gamma(\cM_D;\quad\begin{picture}(42,12)
\put(20,1.5){\makebox(0,0){$\times$}}
\put(36,1.5){\makebox(0,0){$\bullet$}}
\put(1,1.5){\makebox(0,0){$\bullet$}}
\put(1,1.5){\line(1,0){35}}
\put(-3,5){\makebox(0,0)[b]{$\scriptstyle -k-3$}}
\put(20,5){\makebox(0,0)[b]{$\scriptstyle k$}}
\put(36,6){\makebox(0,0)[b]{$\scriptstyle 1$}}
\end{picture}),$$
for $k\leq -3$ and the right hand side has an interpretation in physics as
so-called {\em massless fields of helicity\/} $-1-k/2$
(see e.g.~\cite{EPW} for details).

The main aim of this article is to show that concentration in zero versus top
degree are related by a duality. This will turn out to be useful because the
spectral sequence has simple consequences when concentrated in top degree 
whereas criteria for concentration in degree zero are more easily determined.

\section{The duality}
\begin{theorem}\label{maintheorem}
Let $\kappa_D$ and $\kappa_{\cM_D}$ denote the canonical bundles on $D$ and
$\cM_D$, respectively. Let $d=\dim_{\mathbb{C}}(\mbox{fibres of }\mu)$ and
recall that $s=\dim_{\mathbb{C}}(\mbox{fibres of }\nu)$. Then there are
canonical isomorphisms 
\begin{equation}\label{duality}\nu_*^q\Omega_\mu^p(\kappa_D\otimes E^*)=
\kappa_{\cM_D}\otimes(\nu_*^{s-q}\Omega_\mu^{d-p}(E))^*,\enskip\forall\,
0\leq p\leq d,\,0\leq q\leq s.\end{equation}
The spectral sequence\/ {\rm(\ref{ss})} for the vector bundle $E$ is (strictly)
concentrated in top degree if and only if the corresponding spectral sequence
for $\kappa_D\otimes E^*$ is (strictly) concentrated in degree zero.
\end{theorem}
\begin{proof} Certainly, the last statement follows immediately
from~(\ref{duality}): as $\cM_D$ is contractible and Stein, if 
$\nu_*^q\Omega_\mu^p(E)$ is non-zero then neither is 
$\Gamma(\cM_D;\nu_*^q\Omega_\mu^p(E))$. 

Notice that (\ref{duality}) generalizes Serre duality~\cite{S}. Specifically,
if $D$ is an arbitrary compact manifold, then we may take $\gX_D=D$ and $\cM_D$
to be a point. Then $d=0$, direct images revert to cohomology, and
(\ref{duality}) becomes
$$H^q(D;\cO(\kappa_D\otimes E^*))=H^{s-q}(D;\cO(E))^*.$$
Conversely, Serre duality along the fibers of $\nu$ is the essential 
ingredient in proving (\ref{duality}) as follows. Let $\kappa_{\gX_D}$ denote 
the canonical bundle on~$\gX_D$. Since $\mu$ and $\nu$ are submersions, we can 
write $\kappa_{\gX_D}$ in two different ways:
\begin{equation}\label{two_different_ways}
\kappa_{\gX_D}=\mu^*(\kappa_D)\otimes\kappa_\mu\quad\mbox{and}\quad
\kappa_{\gX_D}=\nu^*(\kappa_{\cM_D})\otimes\kappa_\nu,\end{equation}
where $\kappa_\mu$ and $\kappa_\nu$ are the canonical bundles along the fibers
of $\mu$ and~$\nu$, respectively. Thus, bearing in mind the Hodge isomorphism 
$\Omega_\mu^p=\kappa_\mu\otimes(\Omega_\mu^{d-p})^*$
along the fibers of~$\mu$, we find that
$$\begin{array}{rcl}\nu_*^q\Omega_\mu^p(\kappa_D\otimes E^*)
&=&\nu_*^q\big(\mu^*(\kappa_D)\otimes\Omega_\mu^p\otimes\mu^*(E^*)\big)\\
&=&\nu_*^q\big(\kappa_{\gX_D}\otimes\kappa_\mu^*\otimes
\Omega_\mu^p\otimes\mu^*(E^*)\big)\\
&=&\nu_*^q\big(\nu^*(\kappa_{\cM_D})\otimes\kappa_\nu
\otimes(\kappa_\mu^*\otimes\Omega_\mu^p)\otimes\mu^*(E^*)\big)\\
&=&\kappa_{\cM_D}\otimes\nu_*^q\big(\kappa_\nu\otimes(\Omega_\mu^{d-p})^*
\otimes\mu^*(E^*)\big)\\
&=&\kappa_{\cM_D}\otimes\nu_*^q\big(\kappa_\nu\otimes
(\Omega_\mu^{d-p}\otimes\mu^*(E))^*\big),
\end{array}$$
which may be identified by Serre duality along the fibers of $\nu$ to give
$$\nu_*^q\Omega_\mu^p(\kappa_D\otimes E^*)
=\kappa_{\cM_D}\otimes\big(\nu_*^{s-q}(\Omega_\mu^{d-p}\otimes\mu^*(E))\big)^*
=\kappa_{\cM_D}\otimes(\nu_*^{s-q}\Omega_\mu^{d-p}(E))^*,$$
as required.
\end{proof}    
\section{Applications}\label{applications}
Let us firstly show how Theorem~\ref{maintheorem} yields
(\ref{sampleconcentration}) with minimal calculation. It is clear from
(\ref{relativedeRham}) that strict concentration in degree zero occurs if
$k\geq 0$. Indeed, since 
$\begin{picture}(42,12)
\put(20,1.5){\makebox(0,0){$\times$}}
\put(36,1.5){\makebox(0,0){$\bullet$}}
\put(4,1.5){\makebox(0,0){$\times$}}
\put(4,1.5){\line(1,0){32}}
\put(4,5){\makebox(0,0)[b]{$\scriptstyle -1$}}
\put(20,6){\makebox(0,0)[b]{$\scriptstyle 0$}}
\put(36,6){\makebox(0,0)[b]{$\scriptstyle 0$}}
\end{picture}$
is singular along the fibers of $\nu$ it is also clear that concentration in 
degree zero also occurs when $k=-1$. But now
$$\kappa_D\otimes\big(\begin{picture}(42,12)
\put(20,1.5){\makebox(0,0){$\bullet$}}
\put(36,1.5){\makebox(0,0){$\bullet$}}
\put(4,1.5){\makebox(0,0){$\times$}}
\put(4,1.5){\line(1,0){32}}
\put(4,6){\makebox(0,0)[b]{$\scriptstyle k$}}
\put(20,6){\makebox(0,0)[b]{$\scriptstyle 0$}}
\put(36,6){\makebox(0,0)[b]{$\scriptstyle 0$}}
\end{picture}\big)^*=\begin{picture}(42,12)
\put(20,1.5){\makebox(0,0){$\bullet$}}
\put(36,1.5){\makebox(0,0){$\bullet$}}
\put(4,1.5){\makebox(0,0){$\times$}}
\put(4,1.5){\line(1,0){32}}
\put(4,5){\makebox(0,0)[b]{$\scriptstyle -4$}}
\put(20,6){\makebox(0,0)[b]{$\scriptstyle 0$}}
\put(36,6){\makebox(0,0)[b]{$\scriptstyle 0$}}
\end{picture}\otimes\begin{picture}(42,12)
\put(20,1.5){\makebox(0,0){$\times$}}
\put(36,1.5){\makebox(0,0){$\bullet$}}
\put(4,1.5){\makebox(0,0){$\times$}}
\put(4,1.5){\line(1,0){32}}
\put(4,5){\makebox(0,0)[b]{$\scriptstyle -k$}}
\put(20,6){\makebox(0,0)[b]{$\scriptstyle 0$}}
\put(36,6){\makebox(0,0)[b]{$\scriptstyle 0$}}
\end{picture}=\quad\begin{picture}(42,12)
\put(20,1.5){\makebox(0,0){$\bullet$}}
\put(36,1.5){\makebox(0,0){$\bullet$}}
\put(1,1.5){\makebox(0,0){$\times$}}
\put(1,1.5){\line(1,0){35}}
\put(-3,5){\makebox(0,0)[b]{$\scriptstyle -k-4$}}
\put(20,6){\makebox(0,0)[b]{$\scriptstyle 0$}}
\put(36,6){\makebox(0,0)[b]{$\scriptstyle 0$}}
\end{picture}$$
and Theorem~\ref{maintheorem} tells us that we have strict concentration in top 
degree if and only if $-k-4\geq 0$, which gives $k\leq -4$ as expected. 
Similarly, $-k-4=-1$ if and only if $k=-3$. 

To extend this analysis to vector bundles there are two issue to be overcome. 
The first is that the pullback $\mu^*(\begin{picture}(42,12)
\put(20,1.5){\makebox(0,0){$\bullet$}}
\put(36,1.5){\makebox(0,0){$\bullet$}}
\put(4,1.5){\makebox(0,0){$\times$}}
\put(4,1.5){\line(1,0){32}}
\put(4,6){\makebox(0,0)[b]{$\scriptstyle a$}}
\put(20,6){\makebox(0,0)[b]{$\scriptstyle b$}}
\put(36,6){\makebox(0,0)[b]{$\scriptstyle c$}}
\end{picture})$ is reducible in general. Specifically, 
\begin{equation}\label{pullback}
\raisebox{-20pt}{\makebox[0pt]{$\mu^*(\begin{picture}(42,12)
\put(20,1.5){\makebox(0,0){$\bullet$}}
\put(36,1.5){\makebox(0,0){$\bullet$}}
\put(4,1.5){\makebox(0,0){$\times$}}
\put(4,1.5){\line(1,0){32}}
\put(4,6){\makebox(0,0)[b]{$\scriptstyle a$}}
\put(20,6){\makebox(0,0)[b]{$\scriptstyle b$}}
\put(36,6){\makebox(0,0)[b]{$\scriptstyle c$}}
\end{picture})=\begin{picture}(42,12)
\put(20,1.5){\makebox(0,0){$\times$}}
\put(36,1.5){\makebox(0,0){$\bullet$}}
\put(4,1.5){\makebox(0,0){$\times$}}
\put(4,1.5){\line(1,0){32}}
\put(4,6){\makebox(0,0)[b]{$\scriptstyle a$}}
\put(20,6){\makebox(0,0)[b]{$\scriptstyle b$}}
\put(36,6){\makebox(0,0)[b]{$\scriptstyle c$}}
\end{picture}+
\begin{array}{c}
\begin{picture}(42,12)
\put(20,1.5){\makebox(0,0){$\times$}}
\put(40,1.5){\makebox(0,0){$\bullet$}}
\put(0,1.5){\makebox(0,0){$\times$}}
\put(0,1.5){\line(1,0){40}}
\put(0,6){\makebox(0,0)[b]{$\scriptstyle a+1$}}
\put(20,6){\makebox(0,0)[b]{$\scriptstyle b-2$}}
\put(40,6){\makebox(0,0)[b]{$\scriptstyle c+1$}}
\end{picture}\\
\oplus\\
\begin{picture}(42,12)
\put(20,1.5){\makebox(0,0){$\times$}}
\put(40,1.5){\makebox(0,0){$\bullet$}}
\put(0,1.5){\makebox(0,0){$\times$}}
\put(0,1.5){\line(1,0){40}}
\put(0,6){\makebox(0,0)[b]{$\scriptstyle a+1$}}
\put(20,6){\makebox(0,0)[b]{$\scriptstyle b-1$}}
\put(40,6){\makebox(0,0)[b]{$\scriptstyle c-1$}}
\end{picture}
\end{array}+\enskip
\begin{array}{c}
\begin{picture}(42,12)
\put(20,1.5){\makebox(0,0){$\times$}}
\put(40,1.5){\makebox(0,0){$\bullet$}}
\put(0,1.5){\makebox(0,0){$\times$}}
\put(0,1.5){\line(1,0){40}}
\put(0,6){\makebox(0,0)[b]{$\scriptstyle a+2$}}
\put(20,6){\makebox(0,0)[b]{$\scriptstyle b-4$}}
\put(40,6){\makebox(0,0)[b]{$\scriptstyle c+2$}}
\end{picture}\\
\oplus\\
\begin{picture}(42,12)
\put(20,1.5){\makebox(0,0){$\times$}}
\put(40,1.5){\makebox(0,0){$\bullet$}}
\put(0,1.5){\makebox(0,0){$\times$}}
\put(0,1.5){\line(1,0){40}}
\put(0,6){\makebox(0,0)[b]{$\scriptstyle a+2$}}
\put(20,6){\makebox(0,0)[b]{$\scriptstyle b-3$}}
\put(40,6){\makebox(0,0)[b]{$\scriptstyle c$}}
\end{picture}\\
\oplus\\
\begin{picture}(42,12)
\put(20,1.5){\makebox(0,0){$\times$}}
\put(40,1.5){\makebox(0,0){$\bullet$}}
\put(0,1.5){\makebox(0,0){$\times$}}
\put(0,1.5){\line(1,0){40}}
\put(0,6){\makebox(0,0)[b]{$\scriptstyle a+2$}}
\put(20,6){\makebox(0,0)[b]{$\scriptstyle b-2$}}
\put(40,6){\makebox(0,0)[b]{$\scriptstyle c-2$}}
\end{picture}
\end{array}+
\!\begin{array}{c}\ddots\\ \cdots\\[2pt] \revddots\end{array}\!
+\qquad
\begin{picture}(42,12)
\put(20,1.5){\makebox(0,0){$\times$}}
\put(40,1.5){\makebox(0,0){$\bullet$}}
\put(-7,1.5){\makebox(0,0){$\times$}}
\put(-7,1.5){\line(1,0){47}}
\put(-7,6){\makebox(0,0)[b]{$\scriptstyle a+b+c$}}
\put(20,6){\makebox(0,0)[b]{$\scriptstyle -b-c$}}
\put(40,6){\makebox(0,0)[b]{$\scriptstyle b$}}
\end{picture}\,.$}}\end{equation}
The second is that, even for an irreducible bundle 
$V=\begin{picture}(42,12)
\put(20,1.5){\makebox(0,0){$\times$}}
\put(36,1.5){\makebox(0,0){$\bullet$}}
\put(4,1.5){\makebox(0,0){$\times$}}
\put(4,1.5){\line(1,0){32}}
\put(4,6){\makebox(0,0)[b]{$\scriptstyle a$}}
\put(20,6){\makebox(0,0)[b]{$\scriptstyle b$}}
\put(36,6){\makebox(0,0)[b]{$\scriptstyle c$}}
\end{picture}$ on ${\mathrm{F}}_{1,2}({\mathbb{C}}^4)$, the bundle
$\Omega_\mu^1\otimes V$ may be reducible. For example, these two issues in
combination imply that 
$$\Omega_\mu^1(\begin{picture}(42,12)
\put(20,1.5){\makebox(0,0){$\bullet$}}
\put(36,1.5){\makebox(0,0){$\bullet$}}
\put(4,1.5){\makebox(0,0){$\times$}}
\put(4,1.5){\line(1,0){32}}
\put(4,6){\makebox(0,0)[b]{$\scriptstyle 1$}}
\put(20,6){\makebox(0,0)[b]{$\scriptstyle 0$}}
\put(36,6){\makebox(0,0)[b]{$\scriptstyle 1$}}
\end{picture})=\begin{picture}(42,12)
\put(20,1.5){\makebox(0,0){$\times$}}
\put(36,1.5){\makebox(0,0){$\bullet$}}
\put(4,1.5){\makebox(0,0){$\times$}}
\put(4,1.5){\line(1,0){32}}
\put(4,6){\makebox(0,0)[b]{$\scriptstyle 1$}}
\put(20,5){\makebox(0,0)[b]{$\scriptstyle -2$}}
\put(36,6){\makebox(0,0)[b]{$\scriptstyle 1$}}
\end{picture}\otimes\big(\begin{picture}(42,12)
\put(20,1.5){\makebox(0,0){$\times$}}
\put(36,1.5){\makebox(0,0){$\bullet$}}
\put(4,1.5){\makebox(0,0){$\times$}}
\put(4,1.5){\line(1,0){32}}
\put(4,6){\makebox(0,0)[b]{$\scriptstyle 1$}}
\put(20,6){\makebox(0,0)[b]{$\scriptstyle 0$}}
\put(36,6){\makebox(0,0)[b]{$\scriptstyle 1$}}
\end{picture}+
\begin{picture}(42,12)
\put(20,1.5){\makebox(0,0){$\times$}}
\put(36,1.5){\makebox(0,0){$\bullet$}}
\put(4,1.5){\makebox(0,0){$\times$}}
\put(4,1.5){\line(1,0){32}}
\put(4,6){\makebox(0,0)[b]{$\scriptstyle 2$}}
\put(20,5){\makebox(0,0)[b]{$\scriptstyle -1$}}
\put(36,6){\makebox(0,0)[b]{$\scriptstyle 0$}}
\end{picture}\big)=
\begin{array}{c}\begin{picture}(42,12)
\put(20,1.5){\makebox(0,0){$\times$}}
\put(36,1.5){\makebox(0,0){$\bullet$}}
\put(4,1.5){\makebox(0,0){$\times$}}
\put(4,1.5){\line(1,0){32}}
\put(4,6){\makebox(0,0)[b]{$\scriptstyle 2$}}
\put(20,5){\makebox(0,0)[b]{$\scriptstyle -2$}}
\put(36,6){\makebox(0,0)[b]{$\scriptstyle 2$}}
\end{picture}\\ \oplus\\
\begin{picture}(42,12)
\put(20,1.5){\makebox(0,0){$\times$}}
\put(36,1.5){\makebox(0,0){$\bullet$}}
\put(4,1.5){\makebox(0,0){$\times$}}
\put(4,1.5){\line(1,0){32}}
\put(4,6){\makebox(0,0)[b]{$\scriptstyle 2$}}
\put(20,5){\makebox(0,0)[b]{$\scriptstyle -1$}}
\put(36,6){\makebox(0,0)[b]{$\scriptstyle 0$}}
\end{picture}\end{array}\!\!+\begin{picture}(42,12)
\put(20,1.5){\makebox(0,0){$\times$}}
\put(36,1.5){\makebox(0,0){$\bullet$}}
\put(4,1.5){\makebox(0,0){$\times$}}
\put(4,1.5){\line(1,0){32}}
\put(4,6){\makebox(0,0)[b]{$\scriptstyle 3$}}
\put(20,5){\makebox(0,0)[b]{$\scriptstyle -3$}}
\put(36,6){\makebox(0,0)[b]{$\scriptstyle 1$}}
\end{picture}$$
and the spectral sequence 
$\nu_*^q\Omega_\mu^p(\begin{picture}(42,12)
\put(20,1.5){\makebox(0,0){$\bullet$}}
\put(36,1.5){\makebox(0,0){$\bullet$}}
\put(4,1.5){\makebox(0,0){$\times$}}
\put(4,1.5){\line(1,0){32}}
\put(4,6){\makebox(0,0)[b]{$\scriptstyle 1$}}
\put(20,6){\makebox(0,0)[b]{$\scriptstyle 0$}}
\put(36,6){\makebox(0,0)[b]{$\scriptstyle 1$}}
\end{picture})$ takes the form 
\begin{equation}\label{Theta_sp_seq}
\raisebox{-45pt}{\begin{picture}(300,82)(0,3)
\put(20,20){\makebox(0,0){$\begin{array}{c}\begin{picture}(42,12)
\put(20,1.5){\makebox(0,0){$\times$}}
\put(36,1.5){\makebox(0,0){$\bullet$}}
\put(4,1.5){\makebox(0,0){$\bullet$}}
\put(4,1.5){\line(1,0){32}}
\put(4,6){\makebox(0,0)[b]{$\scriptstyle 1$}}
\put(20,6){\makebox(0,0)[b]{$\scriptstyle 0$}}
\put(36,6){\makebox(0,0)[b]{$\scriptstyle 1$}}
\end{picture}\\ +\\
\begin{picture}(42,12)
\put(20,1.5){\makebox(0,0){$\times$}}
\put(36,1.5){\makebox(0,0){$\bullet$}}
\put(4,1.5){\makebox(0,0){$\bullet$}}
\put(4,1.5){\line(1,0){32}}
\put(4,6){\makebox(0,0)[b]{$\scriptstyle 2$}}
\put(20,5){\makebox(0,0)[b]{$\scriptstyle -1$}}
\put(36,6){\makebox(0,0)[b]{$\scriptstyle 0$}}
\end{picture}\end{array}$}}
\put(140,20){\makebox(0,0){$\begin{array}{c}\begin{picture}(42,12)
\put(20,1.5){\makebox(0,0){$\times$}}
\put(36,1.5){\makebox(0,0){$\bullet$}}
\put(4,1.5){\makebox(0,0){$\bullet$}}
\put(4,1.5){\line(1,0){32}}
\put(4,6){\makebox(0,0)[b]{$\scriptstyle 2$}}
\put(20,5){\makebox(0,0)[b]{$\scriptstyle -2$}}
\put(36,6){\makebox(0,0)[b]{$\scriptstyle 2$}}
\end{picture}\oplus\begin{picture}(42,12)
\put(20,1.5){\makebox(0,0){$\times$}}
\put(36,1.5){\makebox(0,0){$\bullet$}}
\put(4,1.5){\makebox(0,0){$\bullet$}}
\put(4,1.5){\line(1,0){32}}
\put(4,6){\makebox(0,0)[b]{$\scriptstyle 2$}}
\put(20,5){\makebox(0,0)[b]{$\scriptstyle -1$}}
\put(36,6){\makebox(0,0)[b]{$\scriptstyle 0$}}
\end{picture}\\ +\\
\begin{picture}(42,12)
\put(20,1.5){\makebox(0,0){$\times$}}
\put(36,1.5){\makebox(0,0){$\bullet$}}
\put(4,1.5){\makebox(0,0){$\bullet$}}
\put(4,1.5){\line(1,0){32}}
\put(4,6){\makebox(0,0)[b]{$\scriptstyle 3$}}
\put(20,5){\makebox(0,0)[b]{$\scriptstyle -3$}}
\put(36,6){\makebox(0,0)[b]{$\scriptstyle 0$}}
\end{picture}\end{array}$}}
\put(260,20){\makebox(0,0){$\begin{array}{c}\begin{picture}(42,12)
\put(20,1.5){\makebox(0,0){$\times$}}
\put(36,1.5){\makebox(0,0){$\bullet$}}
\put(4,1.5){\makebox(0,0){$\bullet$}}
\put(4,1.5){\line(1,0){32}}
\put(4,6){\makebox(0,0)[b]{$\scriptstyle 3$}}
\put(20,5){\makebox(0,0)[b]{$\scriptstyle -3$}}
\put(36,6){\makebox(0,0)[b]{$\scriptstyle 0$}}
\end{picture}\\ +\\
\begin{picture}(42,12)
\put(20,1.5){\makebox(0,0){$\times$}}
\put(36,1.5){\makebox(0,0){$\bullet$}}
\put(4,1.5){\makebox(0,0){$\bullet$}}
\put(4,1.5){\line(1,0){32}}
\put(4,6){\makebox(0,0)[b]{$\scriptstyle 4$}}
\put(20,5){\makebox(0,0)[b]{$\scriptstyle -4$}}
\put(36,6){\makebox(0,0)[b]{$\scriptstyle 0$}}
\end{picture}\end{array}$}}
\put(45,19){\line(1,0){45}}
\put(235,19){\line(-1,0){45}}
\put(285,19){\vector(1,0){20}}
\put(305,13){$p$}
\put(20,45){\line(0,1){12}}
\put(20,65){\makebox(0,0){$0$}}
\put(140,65){\makebox(0,0){$0$}}
\put(260,65){\makebox(0,0){$0$}}
\put(20,73){\vector(0,1){10}}
\put(12,82){$q$}
\end{picture}}\end{equation}
In particular, it is concentrated in degree zero. This is a general feature as 
follows.
\begin{theorem}\label{zero_concentration}
The spectral sequence for 
$H^r(D;{\mathcal{O}}(\begin{picture}(42,12)
\put(20,1.5){\makebox(0,0){$\bullet$}}
\put(36,1.5){\makebox(0,0){$\bullet$}}
\put(4,1.5){\makebox(0,0){$\times$}}
\put(4,1.5){\line(1,0){32}}
\put(4,6){\makebox(0,0)[b]{$\scriptstyle a$}}
\put(20,6){\makebox(0,0)[b]{$\scriptstyle b$}}
\put(36,6){\makebox(0,0)[b]{$\scriptstyle c$}}
\end{picture}))$ 
associated to the double fibration
\begin{equation}\label{twistors}
\raisebox{-20pt}{\begin{picture}(117,40)
\put(5,5){\makebox(0,0){${\mathbb{CP}}_3^+=D$}}
\put(50,35){\makebox(0,0){$\gX_D$}}
\put(97,5){\makebox(0,0){$\cM_D={\mathbb{M}}^{++}$}}
\put(43,28){\vector(-1,-1){16}}
\put(57,28){\vector(1,-1){16}}
\put(33,23){\makebox(0,0){$\mu$}}
\put(68,23){\makebox(0,0){$\nu$}}
\end{picture}}{}^{\mathrm{open}}\subset
\raisebox{-20pt}{\begin{picture}(100,40)(-13,0)
\put(7,5){\makebox(0,0){${\mathbb{CP}}_3=Z$}}
\put(75,35){\makebox(0,0){$\gX_Z={\mathrm{F}}_{1,2}({\mathbb{C}}^4)$}}
\put(103,5){\makebox(0,0){$\cM_Z={\mathrm{Gr}}_2({\mathbb{C}}^4)$}}
\put(43,28){\vector(-1,-1){16}}
\put(57,28){\vector(1,-1){16}}
\put(33,23){\makebox(0,0){$\mu$}}
\put(68,23){\makebox(0,0){$\nu$}}
\end{picture}}\end{equation}
is strictly concentrated in degree zero if $a\geq 0$.
\end{theorem}
\begin{proof}
Firstly, notice that all the composition factors occurring in (\ref{pullback})
are dominant with respect to the first node if $a\geq 0$. Although clear by
inspection, the underlying reason for this is that the composition factors are 
obtained from the leading term 
$\begin{picture}(42,12)
\put(20,1.5){\makebox(0,0){$\times$}}
\put(36,1.5){\makebox(0,0){$\bullet$}}
\put(4,1.5){\makebox(0,0){$\times$}}
\put(4,1.5){\line(1,0){32}}
\put(4,6){\makebox(0,0)[b]{$\scriptstyle a$}}
\put(20,6){\makebox(0,0)[b]{$\scriptstyle b$}}
\put(36,6){\makebox(0,0)[b]{$\scriptstyle c$}}
\end{picture}$ by adding simple negative roots for 
$\begin{picture}(42,12)
\put(20,1.5){\makebox(0,0){$\bullet$}}
\put(36,1.5){\makebox(0,0){$\bullet$}}
\put(4,1.5){\makebox(0,0){$\times$}}
\put(4,1.5){\line(1,0){32}}
\end{picture}$, namely 
\begin{equation}\label{simplenegative}\begin{picture}(42,12)
\put(20,1.5){\makebox(0,0){$\bullet$}}
\put(36,1.5){\makebox(0,0){$\bullet$}}
\put(4,1.5){\makebox(0,0){$\times$}}
\put(4,1.5){\line(1,0){32}}
\put(4,6){\makebox(0,0)[b]{$\scriptstyle 1$}}
\put(20,5){\makebox(0,0)[b]{$\scriptstyle -2$}}
\put(36,6){\makebox(0,0)[b]{$\scriptstyle 1$}}
\end{picture}\quad\mbox{and}\quad
\begin{picture}(42,12)
\put(20,1.5){\makebox(0,0){$\bullet$}}
\put(36,1.5){\makebox(0,0){$\bullet$}}
\put(4,1.5){\makebox(0,0){$\times$}}
\put(4,1.5){\line(1,0){32}}
\put(4,6){\makebox(0,0)[b]{$\scriptstyle 0$}}
\put(20,6){\makebox(0,0)[b]{$\scriptstyle 1$}}
\put(36,5){\makebox(0,0)[b]{$\scriptstyle -2$}}
\end{picture}\end{equation}
(minus the second two rows 
of the Cartan matrix for ${\mathfrak{sl}}(4,{\mathbb{C}})$) both of which have
a non-negative coefficient over the first node. Secondly, there is the question
of tensoring these composition factors with $\Omega_\mu^p$ from
(\ref{relative_forms}). Each of $\Omega_\mu^p$ is dominant with respect to the 
first node and, more generally,
\begin{equation}\label{more_generally}
\begin{array}{rcl}
\Omega_\mu^0\otimes\begin{picture}(42,12)
\put(20,1.5){\makebox(0,0){$\times$}}
\put(36,1.5){\makebox(0,0){$\bullet$}}
\put(4,1.5){\makebox(0,0){$\times$}}
\put(4,1.5){\line(1,0){32}}
\put(4,6){\makebox(0,0)[b]{$\scriptstyle a$}}
\put(20,6){\makebox(0,0)[b]{$\scriptstyle b$}}
\put(36,6){\makebox(0,0)[b]{$\scriptstyle c$}}
\end{picture}&=\enskip&\begin{picture}(42,12)
\put(20,1.5){\makebox(0,0){$\times$}}
\put(36,1.5){\makebox(0,0){$\bullet$}}
\put(4,1.5){\makebox(0,0){$\times$}}
\put(4,1.5){\line(1,0){32}}
\put(4,6){\makebox(0,0)[b]{$\scriptstyle a$}}
\put(20,6){\makebox(0,0)[b]{$\scriptstyle b$}}
\put(36,6){\makebox(0,0)[b]{$\scriptstyle c$}}
\end{picture}\\
\Omega_\mu^1\otimes\begin{picture}(42,12)
\put(20,1.5){\makebox(0,0){$\times$}}
\put(36,1.5){\makebox(0,0){$\bullet$}}
\put(4,1.5){\makebox(0,0){$\times$}}
\put(4,1.5){\line(1,0){32}}
\put(4,6){\makebox(0,0)[b]{$\scriptstyle a$}}
\put(20,6){\makebox(0,0)[b]{$\scriptstyle b$}}
\put(36,6){\makebox(0,0)[b]{$\scriptstyle c$}}
\end{picture}&=\enskip&\begin{picture}(42,12)
\put(20,1.5){\makebox(0,0){$\times$}}
\put(36,1.5){\makebox(0,0){$\bullet$}}
\put(4,1.5){\makebox(0,0){$\times$}}
\put(4,1.5){\line(1,0){32}}
\put(2,5){\makebox(0,0)[b]{$\scriptstyle a+1$}}
\put(20,5){\makebox(0,0)[b]{$\scriptstyle b-2$}}
\put(38,5){\makebox(0,0)[b]{$\scriptstyle c+1$}}
\end{picture}\enskip\oplus\enskip\begin{picture}(42,12)
\put(20,1.5){\makebox(0,0){$\times$}}
\put(36,1.5){\makebox(0,0){$\bullet$}}
\put(4,1.5){\makebox(0,0){$\times$}}
\put(4,1.5){\line(1,0){32}}
\put(2,5){\makebox(0,0)[b]{$\scriptstyle a+1$}}
\put(20,5){\makebox(0,0)[b]{$\scriptstyle b-1$}}
\put(38,5){\makebox(0,0)[b]{$\scriptstyle c-1$}}
\end{picture}\quad\mbox{(if $c\geq 1$)}\\
\end{array}\end{equation}
\big(noting that it is $%
$ (cf.~(\ref{simplenegative})) that is responsible 
for the second direct summand of $\Omega_\mu^1(%
)$\big). Clearly, all the various composition factors occurring in
$\Omega_\mu^p(%
)$ are dominant with respect to the first node if $a\geq 0$ and,
therefore, all direct images are concentrated in degree zero, as required. 
\end{proof}
\begin{corollary}\label{top_concentration}
The spectral sequence for 
$H^r({\mathbb{CP}}_3^+;{\mathcal{O}}(%
))$ associated to the double fibration {\rm(\ref{twistors})} is 
strictly concentrated in top degree (namely, first degree) if 
$a\leq -4-b-c$.
\end{corollary}
\begin{proof} By standard weight considerations
$$(%
,$$
and we require $-4-a-b-c\geq 0$ in accordance with Theorem~\ref{maintheorem}. 
\end{proof}
\noindent 
Further discussion of this example is postponed until~\S\ref{improved}.

Now let us consider the other example from \S\ref{double_fibrations}, namely
the flag domain~${\mathbb{M}}^{+-}$. As usual, for computational purposes, we
should extend the double fibration to the diagram~(\ref{bigger}). In this case 
we obtain
$$\raisebox{-20pt}{%
\right]\right\}
=\big\{(\Pi_1,\Pi_2)\in
{\mathrm{Gr}}_2({\mathbb{C}}^4)\times{\mathrm{Gr}}_2({\mathbb{C}}^4)\mid
\Pi_1\pitchfork\Pi_2\big\}.$$ 
An additional difficulty in effecting the transform in this case is that,
having taken ${\mathrm{SU}}(2,2)$ to preserve the standard Hermitian
form~(\ref{hermitian}), the usual basepoint for
${\mathrm{Gr}}_2({\mathbb{C}}^4)$ is not in the domain~${\mathbb{M}}^{+-}$.
Instead, as basepoints we may take
$$\left(\left[%
\right]\right\}.$$
On the other hand, we would like to denote the homogeneous vector bundles on
${\mathrm{Gr}}_2({\mathbb{C}}^4)$ by $%
$ as usual. In order to reconcile these two viewpoints, notice
that we may conjugate $Q\subset{\mathrm{SL}}(4,{\mathbb{C}})$ into standard
form: explicitly,
\begin{equation}\label{conjugate}
\left[%
\right]\equiv\widetilde{Q}.\end{equation}
We are therefore confronted with the diagram
\begin{equation}\label{confronted_diagram}
\raisebox{-20pt}{%
}\end{equation}
as the computational key to the double fibration transform. The first 
consequence of this additional feature appears in pulling back an irreducible 
vector bundle from $Z$ to~$\gX_Z$. As already mentioned, we shall
identify $Z$ as $G/\widetilde{Q}$ and write the irreducible bundles thereon as 
$%
$. Irreducible bundles on 
$$\gX_Z=G/(Q\cap K)=
{\mathrm{SL}}(4,{\mathbb{C}})/
\left\{\left[%
\right]\right\},$$
however, are carried by representations of the diagonal subgroup of $Q\cap K$.
Hence, pullback by $\widetilde{\mu}$ may be achieved by restriction to the
subgroup
$$\widetilde{Q}=\left[%
\right]\equiv B$$
followed by conjugation as in (\ref{conjugate}). Explicitly,
$$\widetilde{\mu}^*(%
\big)$$
where $\sigma_2$ denotes the effect of the conjugation (\ref{conjugate}) on
weights. This is a simple Weyl group reflection, the effect of which is 
computed in~\cite{BE}, for example, to obtain
$$\widetilde{\mu}^*(%
.$$
More generally, we may write 
$%
.\end{equation}
The following proposition is almost immediate by inspection. 
\begin{proposition}\label{concentration_by_inspection} The direct images 
$\nu_*^q\widetilde{\mu}^*(%
)$ are strictly concentrated in degree zero if $b\geq 0$.
\end{proposition}
\begin{proof} It remains to observe that, recording the irreducible
homogeneous vector bundles on $\cM_Z$ by irreducible representations of $K$ in 
the usual manner,
\begin{equation}\label{direct_images}
\begin{array}{rclll}
\nu_*(\begin{picture}(42,12)
\put(20,1.5){\makebox(0,0){$\times$}}
\put(36,1.5){\makebox(0,0){$\times$}}
\put(4,1.5){\makebox(0,0){$\times$}}
\put(4,1.5){\line(1,0){32}}
\put(4,6){\makebox(0,0)[b]{$\scriptstyle r$}}
\put(20,6){\makebox(0,0)[b]{$\scriptstyle s$}}
\put(36,6){\makebox(0,0)[b]{$\scriptstyle t$}}
\end{picture})&=&\begin{picture}(42,12)
\put(20,1.5){\makebox(0,0){$\times$}}
\put(36,1.5){\makebox(0,0){$\bullet$}}
\put(4,1.5){\makebox(0,0){$\bullet$}}
\put(4,1.5){\line(1,0){32}}
\put(4,6){\makebox(0,0)[b]{$\scriptstyle r$}}
\put(20,6){\makebox(0,0)[b]{$\scriptstyle s$}}
\put(36,6){\makebox(0,0)[b]{$\scriptstyle t$}}
\end{picture}&\mbox{if }r\geq 0&\mbox{and }t\geq 0\\
\nu_*^1(\begin{picture}(42,12)
\put(20,1.5){\makebox(0,0){$\times$}}
\put(36,1.5){\makebox(0,0){$\times$}}
\put(4,1.5){\makebox(0,0){$\times$}}
\put(4,1.5){\line(1,0){32}}
\put(4,6){\makebox(0,0)[b]{$\scriptstyle r$}}
\put(20,6){\makebox(0,0)[b]{$\scriptstyle s$}}
\put(36,6){\makebox(0,0)[b]{$\scriptstyle t$}}
\end{picture})&=&\enskip\begin{picture}(52,12)
\put(20,1.5){\makebox(0,0){$\times$}}
\put(41,1.5){\makebox(0,0){$\bullet$}}
\put(-1,1.5){\makebox(0,0){$\bullet$}}
\put(-1,1.5){\line(1,0){42}}
\put(-1,6){\makebox(0,0)[b]{$\scriptstyle r$}}
\put(20,6){\makebox(0,0)[b]{$\scriptstyle s+t+1$}}
\put(44,6){\makebox(0,0)[b]{$\scriptstyle -t-2$}}
\end{picture}&\mbox{if }r\geq 0&\mbox{and }t\leq -2\\
\end{array}\end{equation}
and all other direct images vanish;
i.e.,~the usual formul\ae\ \cite{BE} pertain. 
\end{proof}
\begin{lemma}\label{along_mu}
The holomorphic $1$-forms along the fibers of $\mu$ from the
diagram {\rm(\ref{confronted_diagram})} are given by
\begin{equation}\label{Omega_mu}
\Omega_\mu^1=\left(+\begin{array}{c}
\begin{picture}(42,12)
\put(20,1.5){\makebox(0,0){$\times$}}
\put(36,1.5){\makebox(0,0){$\times$}}
\put(4,1.5){\makebox(0,0){$\times$}}
\put(4,1.5){\line(1,0){32}}
\put(4,5){\makebox(0,0)[b]{$\scriptstyle -1$}}
\put(20,5){\makebox(0,0)[b]{$\scriptstyle -1$}}
\put(36,6){\makebox(0,0)[b]{$\scriptstyle 1$}}
\end{picture}\\ \oplus\\
\begin{picture}(42,12)
\put(20,1.5){\makebox(0,0){$\times$}}
\put(36,1.5){\makebox(0,0){$\times$}}
\put(4,1.5){\makebox(0,0){$\times$}}
\put(4,1.5){\line(1,0){32}}
\put(4,6){\makebox(0,0)[b]{$\scriptstyle 1$}}
\put(20,5){\makebox(0,0)[b]{$\scriptstyle -1$}}
\put(36,5){\makebox(0,0)[b]{$\scriptstyle -1$}}
\end{picture}\end{array}+\begin{array}{c}
\begin{picture}(42,12)
\put(20,1.5){\makebox(0,0){$\times$}}
\put(36,1.5){\makebox(0,0){$\times$}}
\put(4,1.5){\makebox(0,0){$\times$}}
\put(4,1.5){\line(1,0){32}}
\put(4,5){\makebox(0,0)[b]{$\scriptstyle -1$}}
\put(20,6){\makebox(0,0)[b]{$\scriptstyle 1$}}
\put(36,6){\makebox(0,0)[b]{$\scriptstyle 1$}}
\end{picture}\\ \oplus\\
\begin{picture}(42,12)
\put(20,1.5){\makebox(0,0){$\times$}}
\put(36,1.5){\makebox(0,0){$\times$}}
\put(4,1.5){\makebox(0,0){$\times$}}
\put(4,1.5){\line(1,0){32}}
\put(4,6){\makebox(0,0)[b]{$\scriptstyle 1$}}
\put(20,6){\makebox(0,0)[b]{$\scriptstyle 1$}}
\put(36,5){\makebox(0,0)[b]{$\scriptstyle -1$}}
\end{picture}\end{array}
\right)\end{equation}
\end{lemma}
\begin{proof}
This is simply a matter of identifying the weights of
${\mathfrak{q}}/({\mathfrak{q}}\cap{\mathfrak{k}})$.
\end{proof}

\begin{theorem}\label{tricky_concentration} The spectral sequence for 
$H^r(D;{\mathcal{O}}(%
))$ 
associated to the double fibration
\begin{equation}\label{cycles_for_M_plus_minus}
\raisebox{-20pt}{%
}\end{equation}
is strictly concentrated in degree zero if $b\geq 0$.
\end{theorem}
\begin{proof} As for Proposition~\ref{concentration_by_inspection}, we should 
inspect the composition factors in 
$\Omega_\mu^p(%
)$ 
and determine, with respect to the first and last nodes, whether they are
dominant or singular (i.e.~whether the integer over that node is non-negative 
or~$-1$, respectively) since,
according to (\ref{direct_images}), such an inspection will determine whether
we have (strict) concentration in degree zero. As the composition factors are 
all line-bundles, this is straightforward arithmetic. If $b\geq 1$, then it is 
easy to check that the leading terms are dominant and the rest are 
mostly dominant but occasionally singular. In this case, the conclusion of the 
theorem is clear. When $b=0$ there are just two exceptions, both of which occur
within $\Omega_\mu^2(
)$. More specifically, Lemma~\ref{along_mu} yields
$$\Omega_\mu^2=\left(\begin{array}{c}
\begin{picture}(42,12)
\put(20,1.5){\makebox(0,0){$\times$}}
\put(36,1.5){\makebox(0,0){$\times$}}
\put(4,1.5){\makebox(0,0){$\times$}}
\put(4,1.5){\line(1,0){32}}
\put(4,6){\makebox(0,0)[b]{$\scriptstyle 0$}}
\put(20,5){\makebox(0,0)[b]{$\scriptstyle -3$}}
\put(36,6){\makebox(0,0)[b]{$\scriptstyle 2$}}
\end{picture}\\ \oplus\\
\begin{picture}(42,12)
\put(20,1.5){\makebox(0,0){$\times$}}
\put(36,1.5){\makebox(0,0){$\times$}}
\put(4,1.5){\makebox(0,0){$\times$}}
\put(4,1.5){\line(1,0){32}}
\put(4,6){\makebox(0,0)[b]{$\scriptstyle 2$}}
\put(20,5){\makebox(0,0)[b]{$\scriptstyle -3$}}
\put(36,6){\makebox(0,0)[b]{$\scriptstyle 0$}}
\end{picture}\end{array}+\cdots
\right)\oplus
\left(\cdots+\begin{array}{c}
\framebox{\begin{picture}(42,12)
\put(20,1.5){\makebox(0,0){$\times$}}
\put(36,1.5){\makebox(0,0){$\times$}}
\put(4,1.5){\makebox(0,0){$\times$}}
\put(4,1.5){\line(1,0){32}}
\put(4,5){\makebox(0,0)[b]{$\scriptstyle -2$}}
\put(20,6){\makebox(0,0)[b]{$\scriptstyle 0$}}
\put(36,6){\makebox(0,0)[b]{$\scriptstyle 2$}}
\end{picture}}\\ \oplus\\
2\begin{picture}(42,12)
\put(20,1.5){\makebox(0,0){$\times$}}
\put(36,1.5){\makebox(0,0){$\times$}}
\put(4,1.5){\makebox(0,0){$\times$}}
\put(4,1.5){\line(1,0){32}}
\put(4,5){\makebox(0,0)[b]{$\scriptstyle 0$}}
\put(20,6){\makebox(0,0)[b]{$\scriptstyle 0$}}
\put(36,6){\makebox(0,0)[b]{$\scriptstyle 0$}}
\end{picture}\\ 
\oplus\\
\framebox{\begin{picture}(42,12)
\put(20,1.5){\makebox(0,0){$\times$}}
\put(36,1.5){\makebox(0,0){$\times$}}
\put(4,1.5){\makebox(0,0){$\times$}}
\put(4,1.5){\line(1,0){32}}
\put(4,6){\makebox(0,0)[b]{$\scriptstyle 2$}}
\put(20,6){\makebox(0,0)[b]{$\scriptstyle 0$}}
\put(36,5){\makebox(0,0)[b]{$\scriptstyle -2$}}
\end{picture}}\end{array}
\right)\oplus
\left(\begin{array}{c}
\begin{picture}(42,12)
\put(20,1.5){\makebox(0,0){$\times$}}
\put(36,1.5){\makebox(0,0){$\times$}}
\put(4,1.5){\makebox(0,0){$\times$}}
\put(4,1.5){\line(1,0){32}}
\put(4,6){\makebox(0,0)[b]{$\scriptstyle 0$}}
\put(20,6){\makebox(0,0)[b]{$\scriptstyle 1$}}
\put(36,6){\makebox(0,0)[b]{$\scriptstyle 2$}}
\end{picture}\\ \oplus\\
\begin{picture}(42,12)
\put(20,1.5){\makebox(0,0){$\times$}}
\put(36,1.5){\makebox(0,0){$\times$}}
\put(4,1.5){\makebox(0,0){$\times$}}
\put(4,1.5){\line(1,0){32}}
\put(4,6){\makebox(0,0)[b]{$\scriptstyle 2$}}
\put(20,6){\makebox(0,0)[b]{$\scriptstyle 1$}}
\put(36,6){\makebox(0,0)[b]{$\scriptstyle 0$}}
\end{picture}\end{array}+\cdots
\right)$$
and the two boxed bundles give rise to 
\begin{equation}\label{two_bad_bundles}%
,\end{equation}
respectively. These two cases require a more delicate analysis, as follows. 
Without loss of generality let us consider the first of them. As can be seen 
from (\ref{Omega_mu}), the line bundle in question arises as a subbundle of 
the rank two vector bundle
$$W\equiv\big(%
.$$
The short exact sequence
$$0\to%
\enskip\to 0$$
induces the exact sequence
$$\begin{array}{ccccccccc}0\to&\nu_*W&\to&\nu_*(\begin{picture}(63,12)
\put(23,1.5){\makebox(0,0){$\times$}}
\put(48,1.5){\makebox(0,0){$\times$}}
\put(4,1.5){\makebox(0,0){$\times$}}
\put(4,1.5){\line(1,0){44}}
\put(4,6){\makebox(0,0)[b]{$\scriptstyle 0$}}
\put(22,6){\makebox(0,0)[b]{$\scriptstyle -a-1$}}
\put(50,6){\makebox(0,0)[b]{$\scriptstyle a+c+2$}}
\end{picture})&\stackrel{\simeq\enskip}{\longrightarrow}&
\nu_*^1(\;\begin{picture}(56,12)
\put(20,1.5){\makebox(0,0){$\times$}}
\put(41,1.5){\makebox(0,0){$\times$}}
\put(4,1.5){\makebox(0,0){$\times$}}
\put(4,1.5){\line(1,0){37}}
\put(4,5){\makebox(0,0)[b]{$\scriptstyle -2$}}
\put(20,6){\makebox(0,0)[b]{$\scriptstyle -a$}}
\put(43,6){\makebox(0,0)[b]{$\scriptstyle a+c+2$}}
\end{picture})
&\to&\nu_*^1W&\to 0\\ &\|&&\|&&\|&&\|\\
0\to&\nu_*W&\to&\begin{picture}(55,12)
\put(23,1.5){\makebox(0,0){$\times$}}
\put(48,1.5){\makebox(0,0){$\bullet$}}
\put(4,1.5){\makebox(0,0){$\bullet$}}
\put(4,1.5){\line(1,0){44}}
\put(4,6){\makebox(0,0)[b]{$\scriptstyle 0$}}
\put(22,6){\makebox(0,0)[b]{$\scriptstyle -a-1$}}
\put(50,6){\makebox(0,0)[b]{$\scriptstyle a+c+2$}}
\end{picture}&\stackrel{\simeq\enskip}{\longrightarrow}&
\begin{picture}(55,12)
\put(23,1.5){\makebox(0,0){$\times$}}
\put(48,1.5){\makebox(0,0){$\bullet$}}
\put(4,1.5){\makebox(0,0){$\bullet$}}
\put(4,1.5){\line(1,0){44}}
\put(4,6){\makebox(0,0)[b]{$\scriptstyle 0$}}
\put(22,6){\makebox(0,0)[b]{$\scriptstyle -a-1$}}
\put(50,6){\makebox(0,0)[b]{$\scriptstyle a+c+2$}}
\end{picture}&\to&\nu_*^1W&\to 0
\end{array}$$
on direct images and we claim that the middle arrows are, in fact, 
isomorphisms, as illustrated. To see this, we trace back the origin of 
$%
$ within $W$ from (\ref{Omega_mu}) and discover that we can write
$$%
)$$
for the natural projection
$$\gX_Z={\mathrm{SL}}(4,{\mathbb{C}})/
\left\{\left[%
\big),$$
all of which vanish owing to the $-1$ over the first node. Since $\nu$ factors 
through $\eta$ it follows that $\nu_*^qW=0$ for all $q$ and the claimed 
isomorphisms follow. Exactly the same isomorphisms arise in computing 
$\nu_*^q\Omega_\mu^2(%
)$ from the direct images of its composition factors and it 
follows that the only two non-zero first direct images arising from the two 
line bundles (\ref{two_bad_bundles}) are, in fact, absorbed into isomorphic 
zeroth direct images. All other contributions are certainly in degree zero.
\end{proof}

\begin{corollary}
The spectral sequence for 
$H^r({\mathbb{M}}^{+-};{\mathcal{O}}(%
))$ associated to the double fibration 
{\rm(\ref{cycles_for_M_plus_minus})} is 
strictly concentrated in top degree (namely, second degree) if 
$b\leq -4-a-c$.
\end{corollary}
\begin{proof} The longest element in the Weyl group of $\widetilde{Q}$ as a
subgroup of the Weyl group of $G$ is~$\sigma_1\sigma_3$, the product of the 
two of simple reflections $\sigma_1$ and~$\sigma_3$. As 
$$\sigma_1\sigma_3(%
$$
it follows that
$$\kappa_{{\mathrm{Gr}}_2({\mathbb{C}}^4)}\otimes(%
.$$
According to Theorems~\ref{maintheorem} and~\ref{tricky_concentration}, we 
require $-4-a-b-c\geq 0$.
\end{proof}
 
\section{Improved duality in the Hermitian holomorphic setting}
\label{improved}
Now let us reconsider the double fibration transform of 
$H^r({\mathbb{CP}}_3^+;\cO(%
))$. According to the spectral sequence (\ref{Theta_sp_seq}), 
this cohomology is transformed to the cohomology of the complex of 
differential operators
$$0\to\begin{array}{c}\begin{picture}(42,12)
\put(20,1.5){\makebox(0,0){$\times$}}
\put(36,1.5){\makebox(0,0){$\bullet$}}
\put(4,1.5){\makebox(0,0){$\bullet$}}
\put(4,1.5){\line(1,0){32}}
\put(4,6){\makebox(0,0)[b]{$\scriptstyle 1$}}
\put(20,6){\makebox(0,0)[b]{$\scriptstyle 0$}}
\put(36,6){\makebox(0,0)[b]{$\scriptstyle 1$}}
\end{picture}\\ +\\
\begin{picture}(42,12)
\put(20,1.5){\makebox(0,0){$\times$}}
\put(36,1.5){\makebox(0,0){$\bullet$}}
\put(4,1.5){\makebox(0,0){$\bullet$}}
\put(4,1.5){\line(1,0){32}}
\put(4,6){\makebox(0,0)[b]{$\scriptstyle 2$}}
\put(20,5){\makebox(0,0)[b]{$\scriptstyle -1$}}
\put(36,6){\makebox(0,0)[b]{$\scriptstyle 0$}}
\end{picture}\end{array}\longrightarrow
\begin{array}{c}\begin{picture}(42,12)
\put(20,1.5){\makebox(0,0){$\times$}}
\put(36,1.5){\makebox(0,0){$\bullet$}}
\put(4,1.5){\makebox(0,0){$\bullet$}}
\put(4,1.5){\line(1,0){32}}
\put(4,6){\makebox(0,0)[b]{$\scriptstyle 2$}}
\put(20,5){\makebox(0,0)[b]{$\scriptstyle -1$}}
\put(36,6){\makebox(0,0)[b]{$\scriptstyle 0$}}
\end{picture}\oplus\begin{picture}(42,12)
\put(20,1.5){\makebox(0,0){$\times$}}
\put(36,1.5){\makebox(0,0){$\bullet$}}
\put(4,1.5){\makebox(0,0){$\bullet$}}
\put(4,1.5){\line(1,0){32}}
\put(4,6){\makebox(0,0)[b]{$\scriptstyle 2$}}
\put(20,5){\makebox(0,0)[b]{$\scriptstyle -2$}}
\put(36,6){\makebox(0,0)[b]{$\scriptstyle 2$}}
\end{picture}\\ +\\
\begin{picture}(42,12)
\put(20,1.5){\makebox(0,0){$\times$}}
\put(36,1.5){\makebox(0,0){$\bullet$}}
\put(4,1.5){\makebox(0,0){$\bullet$}}
\put(4,1.5){\line(1,0){32}}
\put(4,6){\makebox(0,0)[b]{$\scriptstyle 3$}}
\put(20,5){\makebox(0,0)[b]{$\scriptstyle -3$}}
\put(36,6){\makebox(0,0)[b]{$\scriptstyle 0$}}
\end{picture}\end{array}\longrightarrow
\begin{array}{c}\begin{picture}(42,12)
\put(20,1.5){\makebox(0,0){$\times$}}
\put(36,1.5){\makebox(0,0){$\bullet$}}
\put(4,1.5){\makebox(0,0){$\bullet$}}
\put(4,1.5){\line(1,0){32}}
\put(4,6){\makebox(0,0)[b]{$\scriptstyle 3$}}
\put(20,5){\makebox(0,0)[b]{$\scriptstyle -3$}}
\put(36,6){\makebox(0,0)[b]{$\scriptstyle 0$}}
\end{picture}\\ +\\
\begin{picture}(42,12)
\put(20,1.5){\makebox(0,0){$\times$}}
\put(36,1.5){\makebox(0,0){$\bullet$}}
\put(4,1.5){\makebox(0,0){$\bullet$}}
\put(4,1.5){\line(1,0){32}}
\put(4,6){\makebox(0,0)[b]{$\scriptstyle 4$}}
\put(20,5){\makebox(0,0)[b]{$\scriptstyle -4$}}
\put(36,6){\makebox(0,0)[b]{$\scriptstyle 0$}}
\end{picture}\end{array}\to 0$$
on the bounded symmetric domain~${\mathbb{M}}^{++}$. But this complex may
be replaced by an equivalent differential complex as follows. Arguing as 
in~\cite{ER}, the two operators
$$\begin{array}{c}\begin{picture}(42,12)
\put(20,1.5){\makebox(0,0){$\times$}}
\put(36,1.5){\makebox(0,0){$\bullet$}}
\put(4,1.5){\makebox(0,0){$\bullet$}}
\put(4,1.5){\line(1,0){32}}
\put(4,6){\makebox(0,0)[b]{$\scriptstyle 1$}}
\put(20,6){\makebox(0,0)[b]{$\scriptstyle 0$}}
\put(36,6){\makebox(0,0)[b]{$\scriptstyle 1$}}
\end{picture}\\ +\\
\framebox{\begin{picture}(42,12)
\put(20,1.5){\makebox(0,0){$\times$}}
\put(36,1.5){\makebox(0,0){$\bullet$}}
\put(4,1.5){\makebox(0,0){$\bullet$}}
\put(4,1.5){\line(1,0){32}}
\put(4,6){\makebox(0,0)[b]{$\scriptstyle 2$}}
\put(20,5){\makebox(0,0)[b]{$\scriptstyle -1$}}
\put(36,6){\makebox(0,0)[b]{$\scriptstyle 0$}}
\end{picture}}\end{array}\nearrow
\begin{array}{c}\framebox{\begin{picture}(42,12)
\put(20,1.5){\makebox(0,0){$\times$}}
\put(36,1.5){\makebox(0,0){$\bullet$}}
\put(4,1.5){\makebox(0,0){$\bullet$}}
\put(4,1.5){\line(1,0){32}}
\put(4,6){\makebox(0,0)[b]{$\scriptstyle 2$}}
\put(20,5){\makebox(0,0)[b]{$\scriptstyle -1$}}
\put(36,6){\makebox(0,0)[b]{$\scriptstyle 0$}}
\end{picture}}\oplus\begin{picture}(42,12)
\put(20,1.5){\makebox(0,0){$\times$}}
\put(36,1.5){\makebox(0,0){$\bullet$}}
\put(4,1.5){\makebox(0,0){$\bullet$}}
\put(4,1.5){\line(1,0){32}}
\put(4,6){\makebox(0,0)[b]{$\scriptstyle 2$}}
\put(20,5){\makebox(0,0)[b]{$\scriptstyle -2$}}
\put(36,6){\makebox(0,0)[b]{$\scriptstyle 2$}}
\end{picture}\\ +\\
\framebox{\begin{picture}(42,12)
\put(20,1.5){\makebox(0,0){$\times$}}
\put(36,1.5){\makebox(0,0){$\bullet$}}
\put(4,1.5){\makebox(0,0){$\bullet$}}
\put(4,1.5){\line(1,0){32}}
\put(4,6){\makebox(0,0)[b]{$\scriptstyle 3$}}
\put(20,5){\makebox(0,0)[b]{$\scriptstyle -3$}}
\put(36,6){\makebox(0,0)[b]{$\scriptstyle 0$}}
\end{picture}}\end{array}\nearrow
\begin{array}{c}\framebox{\begin{picture}(42,12)
\put(20,1.5){\makebox(0,0){$\times$}}
\put(36,1.5){\makebox(0,0){$\bullet$}}
\put(4,1.5){\makebox(0,0){$\bullet$}}
\put(4,1.5){\line(1,0){32}}
\put(4,6){\makebox(0,0)[b]{$\scriptstyle 3$}}
\put(20,5){\makebox(0,0)[b]{$\scriptstyle -3$}}
\put(36,6){\makebox(0,0)[b]{$\scriptstyle 0$}}
\end{picture}}\\ +\\
\begin{picture}(42,12)
\put(20,1.5){\makebox(0,0){$\times$}}
\put(36,1.5){\makebox(0,0){$\bullet$}}
\put(4,1.5){\makebox(0,0){$\bullet$}}
\put(4,1.5){\line(1,0){32}}
\put(4,6){\makebox(0,0)[b]{$\scriptstyle 4$}}
\put(20,5){\makebox(0,0)[b]{$\scriptstyle -4$}}
\put(36,6){\makebox(0,0)[b]{$\scriptstyle 0$}}
\end{picture}\end{array}$$
are not differential operators at all but, instead, linear isomorphisms of the
vector bundles involved. A diagram chase, performed in~\cite{ER}, shows that we
may cancel these particular bundles to obtain an alternative complex
\begin{equation}\label{alternative_complex}
0\to%
\to 0\end{equation}
calculating the cohomology 
$H^r({\mathbb{CP}}_3^+;\cO(%
))}.$$
Explicit formul\ae\ for the differential operators of
(\ref{alternative_complex}) are given in~\cite{ER}. In particular, they have
orders $1$ and~$2$, respectively.

An alternative route to the same conclusion is as follows. Let us recall that
the spectral sequence in question arises from the double
complex~(\ref{twistors}) and, in particular, from the coupled holomorphic
de~Rham sequence
\begin{equation}\label{coupled_de_Rham}
0\to\mu^*(%
)\to 0\end{equation}
along the fibers of $\mu:{\mathrm{F}}_{1,2}({\mathbb{C}}^4)\to{\mathbb{CP}}_3$.
Bearing in mind that
$$\mu^*(%
$$
as a special case of (\ref{pullback}) and writing out (\ref{coupled_de_Rham})
more explicitly, gives
$$%
$$
and essentially the same cancellation as occurred down on
${\mathrm{Gr}}_2({\mathbb{C}}^4)$ can be seen in this diagram
on~${\mathrm{F}}_{1,2}({\mathbb{C}}^4)$. It means that we may replace the
coupled holomorphic de~Rham sequence along the fibers of $\mu$ by the complex
\begin{equation}\label{projective_BGG_along_mu}
0\!\to\!\mu^{-1}\cO(%
)\!\to\!0\end{equation}
and obtain (\ref{alternative_complex}) directly by zeroth direct image. The
complex (\ref{projective_BGG_along_mu}) is an example of a
Bernstein-Gelfand-Gelfand (BGG) complex. More specifically, the fibers of $\mu$
are intrinsically ${\mathbb{CP}}_2$ and along $\mu$
(\ref{projective_BGG_along_mu}) is a BGG complex on projective space, a simple
general construction for which is given in~\cite{EG}. In summary, although the
filtering (\ref{pullback}) is quite complicated, there is a relatively simple
alternative resolution of $\mu^{-1}\cO(%
)\to 0\end{equation}
obtained by cancelling all but three of the composition factors from the 
coupled de~Rham complex $\Omega_\mu^\bullet(%
)$. This is the general BGG complex along the fibers of~$\mu$. 
It is elementarily constructed in~\cite{EG} and used in \cite{BE} to effect the
double fibration transform in this case. There is a spectral sequence
constructed from (\ref{generalBGG}), which is clearly concentrated in degree
zero if $a\geq 0$. This is the BGG counterpart to
Theorem~\ref{zero_concentration}. Writing $\Delta_\mu^\bullet(E)$ for the BGG 
complex along the fibers of $\mu$ associated to the bundle 
$E=%
$ on $Z={\mathbb{CP}}_3$, firstly we have the following BGG
counterpart to Theorem~\ref{from_BE}.
\begin{theorem}\label{example_of_BGG_ss}
For the double fibration\/ {\rm(\ref{twistors})}, 
and any irreducible homogeneous vector bundle $E$ on~${\mathbb{CP}}_3^+$, 
there is a spectral sequence
\begin{equation}\label{BGG_ss}
E_1^{p,q}=\Gamma({\mathbb{M}}^{++};\nu_*^q\Delta_\mu^p(E))\Longrightarrow 
H^{p+q}({\mathbb{CP}}_3^+;\cO(E)).\end{equation}
\end{theorem}
\begin{proof}
A simple variation on the proof of Theorem~\ref{from_BE}: 
details are in~\cite{BE}.
\end{proof}
Secondly, we have the following BGG counterpart to
Theorem~\ref{zero_concentration}.
\begin{theorem}\label{BGG_zero_concentration}
The spectral sequence\/ {\rm(\ref{BGG_ss})} controlling
$H^r({\mathbb{CP}}_3^+;{\mathcal{O}}(\begin{picture}(42,12)
\put(20,1.5){\makebox(0,0){$\bullet$}}
\put(36,1.5){\makebox(0,0){$\bullet$}}
\put(4,1.5){\makebox(0,0){$\times$}}
\put(4,1.5){\line(1,0){32}}
\put(4,6){\makebox(0,0)[b]{$\scriptstyle a$}}
\put(20,6){\makebox(0,0)[b]{$\scriptstyle b$}}
\put(36,6){\makebox(0,0)[b]{$\scriptstyle c$}}
\end{picture}))$ 
is strictly concentrated in degree zero if $a\geq 0$.
\end{theorem}
\begin{proof}As already discussed, this is clear from (\ref{generalBGG}).
\end{proof}
Notice that the proof of Theorem~\ref{BGG_zero_concentration} is considerably
more straightforward than the proof of Theorem~\ref{zero_concentration} because
the BGG resolution (\ref{generalBGG}) is considerably more straightforward than
the coupled de~Rham complex obtained by combining (\ref{pullback})
with~(\ref{more_generally}). Furthermore, the BGG complex on projective space 
may be constructed without too much difficulty~\cite{EG}.  

The advantages of using the BGG complex in this setting continue with BGG
counterparts to Theorem~\ref{maintheorem} and
Corollary~\ref{top_concentration}, combined as follows.
\begin{proposition} Let $\kappa_{{\mathbb{CP}}_3}=\begin{picture}(42,12)
\put(20,1.5){\makebox(0,0){$\bullet$}}
\put(36,1.5){\makebox(0,0){$\bullet$}}
\put(4,1.5){\makebox(0,0){$\times$}}
\put(4,1.5){\line(1,0){32}}
\put(4,5){\makebox(0,0)[b]{$\scriptstyle -4$}}
\put(20,6){\makebox(0,0)[b]{$\scriptstyle 0$}}
\put(36,6){\makebox(0,0)[b]{$\scriptstyle 0$}}
\end{picture}$ and 
$\kappa_{{\mathrm{Gr}}_2({\mathbb{C}}^4)}=\begin{picture}(42,12)
\put(20,1.5){\makebox(0,0){$\times$}}
\put(36,1.5){\makebox(0,0){$\bullet$}}
\put(4,1.5){\makebox(0,0){$\bullet$}}
\put(4,1.5){\line(1,0){32}}
\put(4,6){\makebox(0,0)[b]{$\scriptstyle 0$}}
\put(20,5){\makebox(0,0)[b]{$\scriptstyle -4$}}
\put(36,6){\makebox(0,0)[b]{$\scriptstyle 0$}}
\end{picture}$ denote the canonical bundles on ${\mathbb{CP}}_3$ and 
${\mathrm{Gr}}_2({\mathbb{C}}^4)$, respectively. Then there are canonical 
isomorphisms
\begin{equation}\label{improved_duality}
\nu_*^q\Delta_\mu^p(\kappa_{{\mathbb{CP}}_3}\otimes E^*)=
\kappa_{{\mathrm{Gr}}_2({\mathbb{C}}^4)}
\otimes(\nu_*^{1-q}\Delta_\mu^{2-p}(E))^*,\enskip\forall\,
0\leq p\leq 2,\,0\leq q\leq 1.\end{equation}
The spectral sequence\/ {\rm(\ref{BGG_ss})} for the vector bundle $E$ is
(strictly) concentrated in top degree if and only if the corresponding spectral
sequence for $\kappa_{{\mathbb{CP}}_3}\otimes E^*$ is (strictly) concentrated 
in degree zero. The spectral sequence\/ {\rm(\ref{BGG_ss})} is concentrated 
in top degree (namely, first degree) if $a\leq-4-b-c$ and in this case yields
an isomorphism between 
$H^r({\mathbb{CP}}_3^+;{\mathcal{O}}(\begin{picture}(42,12)
\put(20,1.5){\makebox(0,0){$\bullet$}}
\put(36,1.5){\makebox(0,0){$\bullet$}}
\put(4,1.5){\makebox(0,0){$\times$}}
\put(4,1.5){\line(1,0){32}}
\put(4,6){\makebox(0,0)[b]{$\scriptstyle a$}}
\put(20,6){\makebox(0,0)[b]{$\scriptstyle b$}}
\put(36,6){\makebox(0,0)[b]{$\scriptstyle c$}}
\end{picture}))$ and 
\begin{equation}\label{the_range}
\ker:\Gamma({\mathbb{M}}^{++};\cO(\begin{picture}(64,12)(-10,0)
\put(30,1.5){\makebox(0,0){$\times$}}
\put(50,1.5){\makebox(0,0){$\bullet$}}
\put(4,1.5){\makebox(0,0){$\bullet$}}
\put(4,1.5){\line(1,0){46}}
\put(2,5){\makebox(0,0)[b]{$\scriptstyle -a-2$}}
\put(30,5){\makebox(0,0)[b]{$\scriptstyle a+b+1$}}
\put(50,6){\makebox(0,0)[b]{$\scriptstyle c$}}
\end{picture}))\to
\Gamma({\mathbb{M}}^{++};\cO(\begin{picture}(76,12)(-14,0)
\put(28,1.5){\makebox(0,0){$\times$}}
\put(47,1.5){\makebox(0,0){$\bullet$}}
\put(4,1.5){\makebox(0,0){$\bullet$}}
\put(4,1.5){\line(1,0){43}}
\put(2,5){\makebox(0,0)[b]{$\scriptstyle -a-b-3$}}
\put(28,6){\makebox(0,0)[b]{$\scriptstyle a$}}
\put(49,5){\makebox(0,0)[b]{$\scriptstyle b+c+1$}}
\end{picture})).\end{equation}
\end{proposition}
\begin{proof} For a linear differential operator $\delta:V\to W$ between any 
two vector bundles, there is its formal adjoint 
$\delta^*:\kappa\otimes W^*\to\kappa\otimes V^*$ and the BGG resolution 
$\Delta_\mu^\bullet(E^*)$ is the formal adjoint of 
$\Delta_\mu^\bullet(E)$. For the bundles 
themselves, this is easily verified from~(\ref{generalBGG}). Specifically,
$$\Delta_\mu^p(E^*)=
\kappa_\mu\otimes\Delta_\mu^{2-p}(E)^*,\enskip\forall\,
0\leq p\leq 2,$$
where recall that, as in the proof of Theorem~\ref{maintheorem}, we denote by
$\kappa_\mu(=\begin{picture}(42,12)
\put(20,1.5){\makebox(0,0){$\times$}}
\put(36,1.5){\makebox(0,0){$\bullet$}}
\put(4,1.5){\makebox(0,0){$\times$}}
\put(4,1.5){\line(1,0){32}}
\put(4,6){\makebox(0,0)[b]{$\scriptstyle 2$}}
\put(20,5.5){\makebox(0,0)[b]{$\scriptstyle -3$}}
\put(36,6){\makebox(0,0)[b]{$\scriptstyle 0$}}
\end{picture})$ the canonical bundle along the fibers of~$\mu$. Also notice from 
(\ref{generalBGG}) that for any line bundle
$L=\begin{picture}(42,12)
\put(20,1.5){\makebox(0,0){$\bullet$}}
\put(36,1.5){\makebox(0,0){$\bullet$}}
\put(4,1.5){\makebox(0,0){$\times$}}
\put(4,1.5){\line(1,0){32}}
\put(4,5){\makebox(0,0)[b]{$\scriptstyle k$}}
\put(20,6){\makebox(0,0)[b]{$\scriptstyle 0$}}
\put(36,6){\makebox(0,0)[b]{$\scriptstyle 0$}}
\end{picture}$ on ${\mathbb{CP}}_3$, we have 
$\Delta_\mu^\bullet(L\otimes E)=\mu^*L\otimes\Delta_\mu^\bullet(E)$.
As in the proof 
of Theorem~\ref{maintheorem} we now use (\ref{two_different_ways}) to write
$\mu^*(\kappa_{{\mathbb{CP}}_3})\otimes\kappa_\mu=
\nu^*(\kappa_{{\mathrm{Gr}}_2({\mathbb{C}}^4)})\otimes\kappa_\nu$ and deduce 
that
$$\Delta_\mu^p(\kappa_{{\mathbb{CP}}_3}\otimes E^*)=
\mu^*(\kappa_{{\mathbb{CP}}_3})\otimes\kappa_\mu\otimes\Delta_\mu^{2-p}(E)^*
=\nu^*(\kappa_{{\mathrm{Gr}}_2({\mathbb{C}}^4)})\otimes\kappa_\nu
\otimes\Delta_\mu^{2-p}(E)^*$$
and, therefore, that
$$\nu_*^q\Delta_\mu^p(\kappa_{{\mathbb{CP}}_3}\otimes E^*)=
\kappa_{{\mathrm{Gr}}_2({\mathbb{C}}^4)}\otimes\nu_*^q(\kappa_\nu
\otimes\Delta_\mu^{2-p}(E)^*).$$
Serre duality along the fibers of $\nu$ yields~(\ref{improved_duality}). The
remaining statements in this proposition are straightforward save for the
appearance of the specific differential (\ref{the_range}) operator whose kernel
identifies the range of the double fibration transform. The bundles may be
obtained by direct computation from (\ref{simple_BBW}), (\ref{generalBGG}), and
(\ref{BGG_ss}). In fact, once the bundles are identified, the differential
operator acting between them is characterized by its $G$-invariance~\cite{ER}, 
which it manifestly enjoys.
\end{proof}
Continuing the theme of formal adjoints implicit in (\ref{improved_duality}), it is interesting to note that the operator
$\nu_*^1\Delta_\mu^0(E)\to\nu_*^1\Delta_\mu^1(E)$ of (\ref{the_range}) may also
be obtained as the formal adjoint of
$\nu_*\Delta_\mu^1(\kappa_{{\mathbb{CP}}_3}\otimes E^*)\to
\nu_*\Delta_\mu^2(\kappa_{{\mathbb{CP}}_3}\otimes E^*)$. This alternative
derivation is available more generally~\cite{EW}.

We shall now discuss the generality in which the BGG resolution can be brought
to bear, as above, in order to improve our understanding of the double
fibration transform. In brief, our considerations of the 
de~Rham-based spectral sequence (\ref{ss}) went as follows.
Our viewpoint was to interpret $H^r(D;\cO(E))$ as an analytic object 
on~$\cM_D$. Roughly speaking, the steps were as follows.
\begin{itemize}
\item Establish isomorphisms
$H^r(D;\cO(E))\cong H^r(\gX_D;\mu^{-1}(\cO(E))),\enskip\forall\,r$.
\item Let $\Omega_\mu^p(E)=\Omega_\mu^p\otimes_\cO\cO(\mu^*E)$,
      the holomorphic $p$-forms along the fibers of~$\mu$, coupled with  
      the pullback $\mu^*E$ of $E$ to~$\gX_D$.
\item Combine these steps to establish Theorem~\ref{from_BE}, namely that 
      there is a spectral sequence 
      $E_1^{p,q}=
      \Gamma(\cM_D;\nu_*^q\Omega_\mu^p(E))\implies H^{p+q}(D;\cO(E))$.
\item Interpret the direct images 
      $\nu_*^q\Omega_\mu^p(E) = \cO(V^{(q)})$ in this spectral sequence as
      certain $(\gg,G_0)$-homogeneous vector bundles $V^{(q)}\to \cM_D$\,.
\item If $E\to D$ is {\em sufficiently negative} then the 
      $\nu_*^q\Omega_\mu^p(E)$ are concentrated in degree 
      $s=\dim_{\mathbb{C}}(\text{fibers of }\nu)$, and the spectral sequence 
      collapses to
\item $H^s(D;\cO(E))\cong\ker:\Gamma(\cM_D,\nu_*^s\Omega_\mu^0(E))
      \stackrel{\nabla}{\longrightarrow}
      \Gamma(\cM_D,\nu_*^s\Omega_\mu^1(E))$.
\end{itemize}
The big problem here is to do the computations that result in the required
sufficient negativity. In good cases, this is facilitated by use of
a BGG-based spectral sequence (see Theorem~\ref{example_of_BGG_ss} and \cite{BE}
more generally). The idea is that the BGG resolution introduces considerable
simplification when $\cM_D$ is a {\em bounded symmetric domain}.

So suppose $\cM_D$ is a bounded symmetric domain $G_0/K_0$ and $P$ is the parabolic
subgroup of $G$ such that $\cM_Z = G/P$ is the complex flag manifolds that
is the dual of $\cM_D$. 

The BGG resolution is constructed from the holomorphic de~Rham resolution
essentially as follows.
\begin{itemize}
\item Observe that $\mu$ has fiber $Q/(P\cap Q) = Q^{ss}/(P \cap Q^{ss})$, 
      where `$\scriptstyle ss$' indicates the semisimple part.
\item Let $W_{(\gp\cap \gq)}^\gq$ consist of the minimal length Weyl group 
      elements $w_{(\gp\cap \gq)}^\gq$ representing the right cosets 
      $W_{(\gp\cap \gq)} \backslash W_\gq$
\item Define $\Delta_\mu^r(\lambda) = 
      \sum \{\cO_{G/(P\cap Q)}(w_{(\gp\cap \gq)}^\gq.\lambda) \mid
      \text{ length } \ell(w_{(\gp\cap \gq)}^\gq) = r\}$, where $\lambda$ 
      is the extremal weight corresponding to the homogeneous bundle~$E$. 
\end{itemize}
The BGG counterpart to the de~Rham-based spectral sequence is the following
result (generalizing Theorem~\ref{example_of_BGG_ss}): If $E \to D$ is an
irreducible $G_0$-homogeneous holomorphic vector bundle then there is a
spectral sequence
\begin{equation}\label{full_blown_BGG_ss}
E_1^{p,q}=\Gamma(\cM_D;\nu_*^q\Delta_\mu^p(E))\implies
H^{p+q}(D;\cO(E)).
\end{equation}
See~\cite{BE} for this theorem and for computation of the 
homogeneous bundles $\Delta_\mu^p(E)$.
The point is that all flag domain computations can be 
carried out for the simpler (compactified) correspondence of flag manifolds. 

The case of flag domains $D$ for which $\cM_D$ is a bounded symmetric domain
$G_0/K_0$ is known as the {\em Hermitian holomorphic} case: see~\cite{FHW} for
details and the rough division of cycle spaces into various other cases. The
double fibration (\ref{twistors}) is the prototype of this sort of cycle space
and is the basic correspondence of {\em twistor theory}~\cite{HT}. This is the
case in which the BGG resolution can be used as indicated above. The details
will appear elsewhere \cite{EW} but here is the final conclusion regarding
sufficient negativity of the bundle $E\to D$ in this case.
\begin{theorem}\label{besttheorem} In the Hermitian holomorphic case the
spectral sequences\/ {\rm (\ref{ss})} and\/ {\rm(\ref{full_blown_BGG_ss})} are
concentrated in degree zero if the highest weight for $E$ is dominant for~$G$.
These spectral sequences are concentrated in top degree if the highest weight
for $\kappa_D\otimes E^*$ is dominant for~$G$.
\end{theorem}
Of course, for this theorem to be useful we need an effective way to compute 
the highest weights of $\kappa_D$ and $E^*$ from that of~$E$. An efficient 
algorithm for this purpose is also given in~\cite{EW}.

\section{Outlook}

As sketched in \S\ref{improved}, we are able to say in the Hermitian
holomorphic case, when the spectral sequence concentrates in top degree. An
effective criterion is given in Theorem~\ref{besttheorem}. Furthermore, by
using the BGG resolution and the improved spectral
sequence~(\ref{full_blown_BGG_ss}) we are able to identify the range of the
transform quite explicitly: the details for the twistor correspondence are in
\S\ref{improved} and in general will appear in~\cite{EW}.

Another important case is where $\cB=G_0/K_0$ is a bounded symmetric domain but
the cycle space $\cM_D$ is not $\cB$ itself but rather $\cB\times\overline\cB$.
Our double fibration (\ref{cycles_for_M_plus_minus}) falls into this category,
generally known as the {\em Hermitian non-holomorphic case}~\cite{FHW}. In this
case, the duality of Theorem~\ref{maintheorem} means that we need only
determine when the spectral sequence is concentrated in degree zero in order to
have an effective criterion for sufficient negativity. One is confronted by
diagrams generalizing (\ref{confronted_diagram}), which may be approached as
in~\S\ref{applications}. The analogue of
Proposition~\ref{concentration_by_inspection} follows without too much
difficulty. The analogue of Theorem~\ref{tricky_concentration} is awkward but,
nevertheless, we conjecture that Theorem~\ref{besttheorem} also holds in the
Hermitian non-holomorphic case.

\bibliographystyle{amsplain}

\end{document}